\documentclass{article}

\usepackage{graphicx}
\usepackage{amsmath}
\usepackage{amssymb}
\usepackage{authblk}
\usepackage{amsthm}

\usepackage{lscape}
\usepackage{pdflscape}

\usepackage{stmaryrd}
\usepackage{mathdots}

\usepackage[geometry]{ifsym}
\usepackage{cmll}

\usepackage{comment}
\usepackage{amscd}
\usepackage{amssymb}

\usepackage{mathrsfs}

\usepackage{color}
\usepackage{xspace}
\usepackage{bussproofs}
\EnableBpAbbreviations

\usepackage{mathtools}

\usepackage{bpextra}

\usepackage{enumerate}

\usepackage{centernot}

\usepackage{hyperref}

\usepackage{cleveref}

\usepackage{caption}

\usepackage{multirow}

\usepackage{array}
\newcolumntype{C}[1]{>{\centering\arraybackslash}p{#1}}
\newcolumntype{L}[1]{>{\arraybackslash}p{#1}}

\usepackage{color}
\usepackage{cancel}
\usepackage{verbatim,cmll}
\usepackage{setspace}
\usepackage{lscape}
\usepackage{latexsym,relsize}
\usepackage{multicol}

\usepackage[position=top]{subfig}
\usepackage{tikz}

\usepackage{txfonts}

\usepackage{multicol}
\usetikzlibrary{arrows}
\usetikzlibrary{matrix}
\usetikzlibrary{patterns}
\usetikzlibrary{shapes}

\newtheorem{thm}{Theorem}[section]
\newtheorem{corollary}[thm]{Corollary}
\newtheorem{lemma}[thm]{Lemma}
\newtheorem{prop}[thm]{Proposition}

\newtheorem{notation}[thm]{Notation}

\newtheorem{example}[thm]{Example}
\newtheorem{definition}[thm]{Definition}
\newtheorem{remark}[thm]{Remark}

\newcommand{\FH}{\hat{f}}

\newcommand{\GC}{\check{g}}

\newcommand{\FHS}{\hat{f}^{\,\sharp}}
\newcommand{\FCS}{\check{f}^{\,\sharp}}
\newcommand{\GHF}{\hat{g}^{\,\flat}}
\newcommand{\GCF}{\check{g}^{\,\flat}}

\newcommand{\aatop}{\ensuremath{\top}\xspace}
\newcommand{\abot}{\ensuremath{\bot}\xspace}
\newcommand{\aand}{\ensuremath{\wedge}\xspace}
\newcommand{\aor}{\ensuremath{\vee}\xspace}

\newcommand{\ABOT}{\ensuremath{\check{\bot}}\xspace}

\newcommand{\wbox}{\ensuremath{\Box}\xspace}
\newcommand{\wdia}{\ensuremath{\Diamond}\xspace}

\newcommand{\WBOX}{\ensuremath{\check{\Box}}\xspace}

\renewcommand{\epsilon}{\varepsilon}

\newcommand{\vp}{\overline{p}}

\newcommand{\ox}{\overline{x}}
\newcommand{\oy}{\overline{y}}
\newcommand{\oz}{\overline{z}}

\newcommand{\ophi}{\overline{\phi}}

\newcommand{\nomh}{\mathbf{h}}

\newcommand{\nomi}{\mathbf{i}}
\newcommand{\rnomi}{\textcolor{red}{\nomi}}
\newcommand{\bnomi}{\textcolor{blue}{\nomi}}

\newcommand{\nomj}{\mathbf{j}}
\newcommand{\rnomj}{\textcolor{red}{\nomj}}
\newcommand{\bnomj}{\textcolor{blue}{\nomj}}

\newcommand{\nomk}{\mathbf{k}}

\newcommand{\noml}{\mathbf{l}}

\newcommand{\cnomm}{\mathbf{m}}
\newcommand{\rcnomm}{\textcolor{red}{\cnomm}}
\newcommand{\bcnomm}{\textcolor{blue}{\cnomm}}

\newcommand{\cnomn}{\mathbf{n}}

\newcommand{\cnomo}{\mathbf{o}}

\newcommand{\cnoml}{\mathbf{l}}

\newcommand{\bba}{\mathbb{A}}
\newcommand{\bbA}{\mathbb{A}}
\newcommand{\bbL}{\mathbb{L}}

\newcommand{\AATOP}{\hat{\top}}

\newcommand{\marginnote}[1]{\marginpar{\raggedright\tiny{#1}}}

\tikzset{
	treenode/.style = {align=center, inner sep=0pt, text centered},
	Ske/.style = {treenode, ellipse, double, draw=black,
		minimum width=6pt, thick},
	PIA/.style = {treenode, ellipse, black, draw=black,
		minimum width=6pt},
	Crit/.style = {treenode, rectangle, draw=black,
		minimum width=0.5em, minimum height=0.5em}
}

\def\fCenter{{\mbox{$\ \vdash\ $}}}

\EnableBpAbbreviations

\newcommand{\fns}{\footnotesize}

\newcommand{\commment}[1]{}
\numberwithin{equation}{section}

\newcommand{\cfDLE}{\underline{D.LE}}
\newcommand{\cfDDLE}{\underline{D.DLE}}

\newcommand{\bx}{\textcolor{blue}{x}}

\newcommand{\bu}{\textcolor{blue}{u}}
\newcommand{\bz}{\textcolor{blue}{z}}

\newcommand{\bt}{\textcolor{blue}{t}}
\newcommand{\bone}{\textcolor{blue}{1}}
\newcommand{\bn}{\textcolor{blue}{n}}
\newcommand{\obp}{\overline{\textcolor{blue}{p}}}
\newcommand{\obx}{\overline{\textcolor{blue}{x}}}
\newcommand{\obz}{\overline{\textcolor{blue}{z}}}

\newcommand{\ry}{\textcolor{red}{y}}

\newcommand{\rw}{\textcolor{red}{w}}
\newcommand{\rt}{\textcolor{red}{t}}

\newcommand{\rone}{\textcolor{red}{1}}
\newcommand{\rn}{\textcolor{red}{n}}
\newcommand{\orq}{\overline{\textcolor{red}{q}}}
\newcommand{\ory}{\overline{\textcolor{red}{y}}}
\newcommand{\orw}{\overline{\textcolor{red}{w}}}
\newcommand{\ba}{\textcolor{blue}{\alpha}}
\newcommand{\oba}{\overline{\textcolor{blue}{\alpha}}}
\newcommand{\bpsi}{\textcolor{blue}{\psi}}

\newcommand{\bgamma}{\textcolor{blue}{\gamma}}

\newcommand{\bsigma}{\textcolor{blue}{\sigma}}

\newcommand{\bpi}{\textcolor{blue}{\pi}}
\newcommand{\obu}{\overline{\textcolor{blue}{u}}}
\newcommand{\orv}{\overline{\textcolor{red}{v}}}

\newcommand{\orb}{\overline{\textcolor{red}{\beta}}}
\newcommand{\rxi}{\textcolor{red}{\xi}}

\newcommand{\rdelta}{\textcolor{red}{\delta}}

\newcommand{\rtau}{\textcolor{red}{\tau}}

\newcommand{\rsigma}{\textcolor{red}{\sigma}}

\newcommand{\cceq}{\coloneqq}

\newcommand{\OSI}{\overline{\Sigma}}

\newcommand{\OXI}{\overline{\Xi}}
\newcommand{\ol}[1]{\overline{#1}}

\title{Inception Display Calculi}
\author[1]{Andrea De Domenico}
\author[2]{Giuseppe Greco}
\author[2,3]{Alessandra Palmigiano}

\affil[1]{\small IMDEA Software Institute, Madrid, Spain}
\affil[2]{\small Vrije Universiteit Amsterdam}
\affil[3]{\small Department of Mathematics and Applied Mathematics, U.~of Johannesburg, South Africa}

\date{ \ } 

\begin{document}
\maketitle

\begin{abstract}
Display calculi were introduced by Nuel Belnap in \cite{belnap1982display} as a natural extension of Gentzen's sequent calculi, as a uniform and modular framework capable of encompassing broad classes of logics. In \cite{ucaaptt}, the properly displayable (D)LE-logics are syntactically characterized as the logics axiomatised by {\em analytic inductive axioms}  for any signature. We extend the framework of proper display calculi for LE-logics to include axiomatic extensions with axioms that are {\em inductive} but not necessarily analytic inductive. This class of axioms covers and properly extends all Sahlqvist axioms. The present framework takes inspiration from Schroeder-Heister's calculus of Higher-Level Rules \cite{Psh14} and captures the whole acyclic fragment of the substructural hierarchy \cite{ciabattoni2008axioms} when generalized to arbitrary signatures. We apply unified correspondence theory and the algorithm ALBA to uniformly generate analytic rules for the aforementioned axiomatic extensions.

\end{abstract}




\section{Introduction}
\label{sec: intro}
Display calculi are a celebrated proof-theoretic framework introduced by Belnap \cite{belnap1982display}. They extend Gentzen’s sequent calculi by enriching the language with {\em structural connectives}, that is, syntactic constructors generating well-formed \emph{structures} occurring in sequents. Formulas are treated as atomic structures. In Gentzen’s original calculi, the only structural connectives are the commas separating formulas in the antecedent and consequent of sequents; display calculi generalise this idea by allowing a much richer structural language.

A central design principle of display calculi is the systematic separation between {\em logical rules} and {\em structural rules}. Logical rules introduce logical connectives according to a uniform scheme determined by their arity and polarity (tonicity) in each coordinate. Moreover, each logical connective is systematically associated with suitable structural connectives, which act as its “proxies” at the structural level. In this way, logical behaviour is reflected in the architecture of the structural language.

The choice of structural connectives is constrained by the requirement that the calculus satisfy the \emph{display property}: for every derivable sequent $s = X \vdash Y$, and for every substructure $Z$ occurring in $s$, there exists a sequent $s'$ interderivable with $s$ such that either $s' = Z \vdash W$ or $s' = W \vdash Z$. Thus every substructure of a derivable sequent can be “displayed” in isolation on one side of the sequent. The interderivability of $s$ and $s'$ is ensured by structural rules known as \emph{display postulates}, which encode adjointness or residuation conditions for each structural connective in every coordinate.

Conceptually, the display property guarantees that the operational core of a rule can always be isolated, which in turn allows for a uniform formulation of logical rules and of the cut rule. This structural transparency underlies Belnap’s general cut-elimination theorem for display calculi.

Different strategies exist for choosing structural connectives. In \cite{belnap1982display}, Belnap kept their number minimal. By contrast, \cite{Gore1} introduced a structural proxy for each logical connective. In the present paper, we follow the latter approach: for every logical connective, we introduce a corresponding structural connective and close the structural language under all adjoints or residuals in every coordinate.

Structural rules manipulate only structural connectives and encode the specific features of the logic under consideration—properties that, in a Hilbert-style presentation, would typically be expressed by axioms. The class of admissible structural rules is precisely characterised so that, when added to a base display calculus enjoying cut elimination, this property is preserved.

{\em Proper} display calculi, identified by Wansing \cite[Section 4.1]{wansing1998displaying}, form a subclass in which every structural rule is closed under uniform substitution. This additional constraint yields a robust meta-theory and a high degree of modularity.

In \cite{Gore98Gaggles}, it is shown that every basic LE-logic can be captured by a cut-free basic display calculus. In \cite{ucaaptt}, two open questions in the field are answered positively:
\begin{itemize}
\item[(i)] Let $\mathbf{L}$ be a logic in the language $\mathcal{L}$ and let $\mathbf{D.L}$ be a display calculus with logical language $\mathcal{L}$ and structural language $\mathcal{L}^\ast$, where $\mathcal{L}^\ast \supset \mathcal{L}$, $\mathcal{L}^\ast \setminus \mathcal{L}$ contains only adjoints or residuals of connectives in $\mathcal{L}$, the operations interpreting the connectives in $\mathcal{L}^\ast$ are fully residuated while those interpreting the connectives in $\mathcal{L}$ are not, and such that $\mathbf{D.L}$ derives all the theorems of $\mathbf{L}$. Then the logic of $\mathbf{D.L}$ is a conservative extension of $\mathbf{L}$.
\item[(ii)] Let $\mathbf{L}$ be a logic in the language $\mathcal{L}$ presented via a display calculus $\mathbf{D.L}$, and let $\mathbf{L'} = \mathbf{L} \cup \Sigma$, where $\Sigma$ is a set of \emph{analytic inductive} axioms in $\mathcal{L}$. Then a display calculus $\mathbf{D.L'}$ capturing $\mathbf{L'}$ can be automatically generated.
\end{itemize}

In \cite{LEframes}, these results are generalised to arbitrary LE-logics in arbitrary signatures, including semantic cut elimination. A comparison between the basic notions and terminology of gaggle theory and order theory is provided in \cite[Appendix B]{LEframes}.

The algorithmic generation of analytic rules has become a central theme in structural proof theory. The methodology has been refined and extended beyond display calculi by several authors; see \cite{ChnGrePalTzi21} for a comprehensive survey.

Proper \emph{multi-type} display calculi extend Belnap’s framework by introducing multiple types, allowing distinct structural environments to be handled separately. They retain the defining features of proper display calculi—closure under uniform substitution (relativised to each type) and modular treatment of analytic inductive axioms—while significantly increasing expressive power.

Although general characterisation results limit the scope of single-type proper display calculi, the multi-type setting considerably broadens applicability. In particular, it allows one to treat logics that are not properly displayable in a single-type language, thereby expanding the range of systems that can be captured within the display-calculus framework.

This includes well-known logical systems such as inquisitive logic \cite{inquisitive}, dynamic epistemic logic, propositional dynamic logic, semi-De Morgan logic and its extensions \cite{SDM}, bilattice logic \cite{greco2017bilattice}, non-normal and conditional logics \cite{chen2022non}, and logics of rough algebras \cite{GJMPT19,ProperMulti-TypeDisplayCalculiforRoughAlgebras}. The methodology also provides a principled basis for designing new families of logical systems, such as those introduced in \cite{bilkova2018logic}.

Finally, a formal connection has been established between correspondence phenomena \cite{Sah75} and display calculi \cite{belnap1982display}, via unified correspondence theory \cite{CoGhPa14}. One important outcome is the development of proper display calculi for LE-logics \cite{conradie2019algorithmic}, together with a systematic method for transforming analytic inductive inequalities into structural rules that can be modularly added to a base calculus without compromising cut elimination \cite{ucaaptt,ChnGrePalTzi21}.

As mentioned above, in previous years, a formal connection between correspondence phenomena \cite{Sah75} and the theory of display calculi \cite{belnap1982display} has been established, applying results and insights from unified correspondence theory \cite{CoGhPa14}. One of the consequences of this connection is the development of proper display calculi for the class of LE-logics \cite{conradie2019algorithmic}, together with a method for converting a broad class of axioms (the class of  analytic inductive inequalities) into rules that can be modularly added to the base calculus without disrupting the eliminability of the cut rule \cite{ucaaptt,ChnGrePalTzi21}.

In this paper, we generalize the framework of (single-type) proper display calculi for LE-logics to include axiomatic extensions with axioms that are inductive \cite{conradie2019algorithmic} but not necessarily analytic inductive, thus extending the class of axioms that can be converted into analytic rules. This class covers and properly extends all Sahlqvist axioms. A semantic analysis of the first-order correspondents of inductive inequalities suggests an approach that is similar in nature to that of Schroeder-Heister's Calculus of Higher-level Rules \cite{Psh14}, and captures the acyclic portion of the substructural hierarchy \cite{CiRa14}. 

Our approach aims at achieving the same goal as  Negri's method for generating labelled calculi with systems of rules \cite{Negri14}, in the uniform setting of LE-logics. 
We apply the algorithm ALBA \cite{conradie2019algorithmic} to uniformly generate analytic rules for the previously mentioned inductive axiomatic extensions. The resulting \emph{inception display calculi} framework is  named after Christopher Nolan's 2010 movie,  set in an interpersonal dream space in which multiple dreams can be nested deep into other dreams. Analogously, the framework of inception display calculi allows for derivations to occur nested within other derivations,  in the form of higher-order assumptions. We  show cut elimination for these calculi  via a Belnap-style metatheorem, under the assumption that, in each derivation,   the `nesting' relation (referred to as  {\em inception depth}) is well-founded.

{\em Structure of the paper. }In Section \ref{sec: prelim}, we gather preliminary notions on $\mathrm{LE}$-logics, their syntax, algebraic semantics, inductive LE-axioms, the workings of ALBA,  and display calculi for LE-logics. In Section \ref{sec:inception}, we introduce the framework of inception display calculi and show some examples. In Section \ref{sec:cut_elimination}, we prove a Belnap-style cut-elimination metatheorem for the family of proper inception display calculi. Finally, in Section \ref{sec:concl}, we summarize the results of this article and collect further research directions.

\section{Preliminaries}
\label{sec: prelim}

This section collects basic facts and results from \cite[Section 2 and Appendix A]{LEframes} and \cite[Section 2]{ChnGrePalTzi21}.

\subsection{LE-logics and their algebraic semantics}
\label{ssec: LE-logics}

 \paragraph{Language and axiomatization of basic normal LE-logics.}
An {\em order-type} over $n\in \mathbb{N}$ is an $n$-tuple $\epsilon\in \{1, \partial\}^n$. For every order type $\epsilon$, we denote its {\em opposite} order type by $\epsilon^\partial$, that is, $\epsilon^\partial_i = 1$ iff $\epsilon_i=\partial$ for every $1 \leq i \leq n$. For any lattice $\bba$, we let $\bba^1: = \bba$ and $\bba^\partial$ be the dual lattice, that is, the lattice associated with the converse partial order of $\bba$. For any order type $\varepsilon$ over $n$, we let $\bba^\varepsilon: = \Pi_{i = 1}^n \bba^{\varepsilon_i}$.
	
The language $\mathcal{L}_\mathrm{LE}(\mathcal{F}, \mathcal{G})$ (from now on abbreviated as $\mathcal{L}_\mathrm{LE}$) takes as parameters a denumerable set of proposition letters $\mathsf{AtProp}$, elements of which are denoted $p,q,r$, possibly with indexes, and disjoint sets of connectives $\mathcal{F}$ and $\mathcal{G}$.
Each $f\in \mathcal{F}$ (resp.~$g\in \mathcal{G}$) has arity $n_f\in \mathbb{N}$ (resp.\ $n_g\in \mathbb{N}$) and is associated with some order-type $\varepsilon_f$ over $n_f$ (resp.\ $\varepsilon_g$ over $n_g$). The terms (formulas) of $\mathcal{L}_\mathrm{LE}$ are defined recursively as follows: 

\begin{center}
$\varphi \cceq p \mid \bot \mid \top \mid \varphi \wedge \varphi \mid \varphi \vee \varphi \mid f(\varphi_1, \ldots, \varphi_{n_f}) \mid g(\varphi_1, \ldots, \varphi_{n_g})$
\end{center}

\noindent where $p \in \mathsf{AtProp}$, $f \in \mathcal{F}$, $g \in \mathcal{G}$. Terms in $\mathcal{L}_\mathrm{LE}$ will be denoted either by $s,t$, or by lowercase Greek letters such as $\varphi, \psi, \gamma$ etc. 
In the remainder of the paper, when it is clear from the context,
we will often simplify notation and write e.g.\  $n$ for $n_f$ and $\varepsilon_i$ for $\varepsilon_{f,i}$. We also extend the $\{1,\partial\}$-notation to the symbols $\vee,\wedge,\bot,\top,\le,\vdash$ by defining 

\begin{center}
$\vee^\partial=\wedge,\qquad \wedge^\partial=\vee,\qquad \bot^\partial=\top,\qquad \top^\partial=\bot,\qquad {\le^\partial}={\ge},\qquad \alpha\ {\vdash^\partial}\ \beta=\beta\ {\vdash}\ \alpha
$
\end{center}

\noindent while superscript $^1$ denotes the identity map. Therefore, in what follows, we will sometimes write e.g.~$\vee^{\epsilon_i}$ to denote $\vee$ when $\epsilon_i = 1$ and  $\wedge$ when $\epsilon_i = \partial$. In what follows, for every $k \in \mathcal{F} \cup \mathcal{G}$ we use $k(\ophi)[\varphi]_i$ to indicate that the formula $\varphi$ occurs in the $i$-th coordinate of the vector $\ophi$.

For any language $\mathcal{L}_\mathrm{LE} = \mathcal{L}_\mathrm{LE}(\mathcal{F}, \mathcal{G})$, the {\em basic}, or {\em minimal} $\mathcal{L}_\mathrm{LE}$-{\em logic} $\mathbf{L}_\mathrm{LE}$ is a set of sequents $\varphi\vdash\psi$, with $\varphi,\psi\in\mathcal{L}_\mathrm{LE}$, which contains as axioms the following sequents for lattice operations and additional connectives:

\begin{center}
$\bot\vdash p, \ \ \ p\vdash p, \ \ \ p\vdash \top, \ \ \ p\vdash p \vee q, \ \ \ q\vdash p\vee q, \ \ \ p\wedge q\vdash p, \ \ \ p\wedge q \vdash q,$
\end{center}

\begin{center}
$f(\vp)[\bot^{\epsilon_i}]_i \vdash \bot,  \qquad f(\vp)[q\vee^{\epsilon_i} r]_i \vdash f(\vp)[q]_i \vee f(\vp)[r]_i,$
\end{center}

\begin{center}
$\top \vdash g(\vp)[\top^{\epsilon_i}]_i, \qquad g(\vp)[q]_i \wedge g(\vp)[r]_i \vdash g(\vp)[q\wedge^{\epsilon_i} r]_i,$
\end{center}

\noindent and is closed under the following inference rules (note that $\varphi\vdash^\partial\psi$ means $\psi\vdash\varphi$):

\begin{center}
\begin{tabular}{@{}c@{}}
\AXC{$\varphi \vdash \chi$}
\AXC{$\chi \vdash \psi$}
\BIC{$\varphi \vdash \psi$}
\DP
 \ 
\AXC{$\varphi \vdash \psi$}
\UIC{$\varphi(\chi/p) \vdash \psi(\chi/p)$}
\DP
 \ 
\AXC{$\chi \vdash \varphi$}
\AXC{$\chi \vdash \psi$}
\BIC{$\chi \vdash \varphi \wedge \psi$}
\DP
 \ 
\AXC{$\varphi \vdash \chi$}
\AXC{$\psi \vdash \chi$}
\BIC{$\varphi \vee \psi \vdash \chi$}
\DP
 \\
 
\rule[0mm]{0mm}{8mm}\AXC{$\varphi \vdash^{\epsilon_{f,i}} \psi$}
\UIC{$f(\vp)[\varphi]_i \vdash f(\vp)[\psi]_i$}
\DP

\qquad\qquad

\AXC{$\varphi \vdash^{\epsilon_{g,i}} \psi$}
\UIC{$g(\vp)[\varphi]_i \vdash g(\vp)[\psi]_i$}
\DP
 \\
\end{tabular}
\end{center}


In a basic language $\mathcal{L}_{\mathrm{LE}}(\mathcal{F},\mathcal{G})$, the elements of $\mathcal{F}\cup\mathcal{G}$ are usually mutually independent. However, in some cases,  some   connectives in $\mathcal{F}\cup\mathcal{G}$ are one another's residuals (or (dual) Galois-residuals) in some coordinates.\footnote{In what follows, we will typically not make a difference between co-variant and contravariant residuation, and refer to any of these as `residuals'.} In particular, given $f\in\mathcal{F}$ or $g\in\mathcal{G}$ we  have $f^\sharp_i\in\mathcal{G}$ if $\varepsilon_{f,i} = 1$ or $g^\flat_i\in\mathcal{F}$ if $\varepsilon_{g,i} = 1$, and $f^\sharp_i\in\mathcal{F}$ if $\varepsilon_{f,i} = \partial$ or $g^\flat_i\in\mathcal{G}$ if $\varepsilon_{g,i} = \partial$, the order-type of which are (i) $\epsilon_{f_i^\sharp,i} = \epsilon_{f,i}$ and $\epsilon_{f_i^\sharp,j} = (\epsilon_{f,j})^{\epsilon_{f,i}^\partial}$ for any $j\neq i$, and (ii) $\epsilon_{g_i^\flat,i} = \epsilon_{g,i}$ and $\epsilon_{g_i^\flat,j} = (\epsilon_{g,j})^{\epsilon_{g,i}^\partial}$ for any $j\neq i$. For instance, if $f$ and $g$ are binary connectives such that $\varepsilon_f = (1, \partial)$ and $\varepsilon_g = (\partial, 1)$, then $\varepsilon_{f^\sharp_1} = (1, 1)$, $\varepsilon_{f^\sharp_2} = (1, \partial)$, $\varepsilon_{g^\flat_1} = (\partial, 1)$ and $\varepsilon_{g^\flat_2} = (1, 1)$.\footnote{Note that this notation depends on the connective which is taken as primitive, and needs to be carefully adapted to well known cases. For instance, consider the  `fusion' connective $\circ$ (which, when denoted  as $f$, is such that $\varepsilon_f = (1, 1)$). Its residuals
$f_1^\sharp$ and $f_2^\sharp$ are commonly denoted $/$ and
$\backslash$ respectively. However, if $\backslash$ is taken as the primitive connective $g$, then $g_2^\flat$ is $\circ = f$, and
$g_1^\flat(x_1, x_2): = x_2/x_1 = f_1^\sharp (x_2, x_1)$. This example shows
that, when identifying $g_1^\flat$ and $f_1^\sharp$, the conventional order of the coordinates is not preserved, and depends on which connective
is taken as primitive.} The intended interpretation of the $n_f$-ary connective $f^\sharp_i$ is  
the right residual of $f\in\mathcal{F}$ in its $i$th coordinate if $\varepsilon_{f,i} = 1$ (resp.\ its dual Galois-residual if $\varepsilon_{f,i} = \partial$). The intended interpretation of the  $n_g$-ary connective $g^\flat_i$ is the left residual of $g\in\mathcal{G}$ in its $i$th coordinate if $\varepsilon_{g,i} = 1$ (resp.\ its Galois-residual if $\varepsilon_{g,i} = \partial$).

%
If so, $\mathbf{L}_{\mathrm{LE}}$ is augmented with the invertible rules 

\begin{center}
{\centerline{
\begin{tabular}{cc}
\AXC{$f(\ophi)[\varphi]_i \fCenter \psi$}
\doubleLine
\UIC{$\varphi \fCenter^{\epsilon_{f,i}} f^\sharp_i(\ophi)[\psi]_i$}
\DP
\qquad & \qquad
\AXC{$\varphi \fCenter g(\ophi)[\psi]_i$}
\doubleLine
\UIC{$g^\flat_i(\ophi)[\varphi]_i \fCenter^{\epsilon_{g,i}} \psi$}
\DP \\
\end{tabular}
}}
\end{center}


Any given language $\mathcal{L}_\mathrm{LE} = \mathcal{L}_\mathrm{LE}(\mathcal{F}, \mathcal{G})$ can be associated with the language $\mathcal{L}_\mathrm{LE}^* = \mathcal{L}_\mathrm{LE}(\mathcal{F}^*, \mathcal{G}^*)$, where $\mathcal{F}^*\supseteq \mathcal{F}$ and $\mathcal{G}^*\supseteq \mathcal{G}$ are obtained by expanding $\mathcal{L}_\mathrm{LE}$ with the residuals of each connective in each coordinate. Then, the logic $\mathbf{L}_{\mathrm{LE}}$ can be expanded to the minimal fully residuated $\mathcal{L}_\mathrm{LE}^\ast$-logic $\mathbf{L}_\mathrm{LE}^*$, by adding the corresponding residuation rules.  
The logic $\mathbf{L}_\mathrm{LE}^*$ is a conservative extension of $\mathbf{L}_\mathrm{LE}$ (cf.~\cite[Theorem 2.4]{ChnGrePalTzi21}), i.e.~every $\mathcal{L}_\mathrm{LE}$-sequent $\varphi\vdash\psi$ is derivable in $\mathbf{L}_\mathrm{LE}$ iff $\varphi\vdash\psi$ is derivable in $\mathbf{L}_\mathrm{LE}^*$. 

\paragraph{\textbf{LE-algebras.}} For any tuple $(\mathcal{F}, \mathcal{G})$ of disjoint sets of function symbols as above, a {\em  lattice expansion} (abbreviated as LE) is a tuple $\bba = (\bbL, \mathcal{F}^\bbA, \mathcal{G}^\bbA)$ such that $\bbL$ is a lattice, $\mathcal{F}^\bbA = \{f^\bbA\mid f\in \mathcal{F}\}$ and $\mathcal{G}^\bbA = \{g^\bbA\mid g\in \mathcal{G}\}$, such that every $f^\bbA\in\mathcal{F}^\bbA$ (resp.\ $g^\bbA\in\mathcal{G}^\bbA$) is an $n_f$-ary (resp.\ $n_g$-ary) operation on $\bbA$. We will often simplify notation and write e.g.\ $f$ for $f^\bbA$. Such an operation $f$ (resp.~$g$)  is an {\em operator} if for every $1\leq i\leq n$,
  \centerline{
$f(\vp)[q\vee^{\epsilon_i} r]_i = f(\vp)[q]_i \vee f(\vp)[r]_i \quad \text{ and }\quad
g(\vp)[q\wedge^{\epsilon_i} r]_i = g(\vp)[q]_i \wedge g(\vp)[r]_i,$
}
and it is {\em normal} if $f(\vp)[\bot^{\epsilon_i}]_i = \bot \text{ and } g(\vp)[\top^{\epsilon_i}]_i = \top.$

A normal LE as above is {\em complete} if, in addition, $\mathbb{L}$ is a complete lattice and the operation corresponding to each $f\in \mathcal{F}$ (resp.~$g\in \mathcal{G}$) is coordinate-wise completely join-preserving (resp.~meet-preserving) when regarded as a map $f^\bba: \bba^{\varepsilon_f}\to \bba$ (resp.\ $g^\bba: \bba^{\varepsilon_g}\to \bba$). By well known order-theoretic facts (cf.~\cite[Proposition 7.34]{DaPr}), a complete normal LE is also {\em completely residuated}, i.e.~the  residuals $f^\sharp_i$ (resp.~$g^\flat_i$) 
in each coordinate $i$ exist of the operations corresponding to  every $f\in \mathcal{F}$ (resp.~$g\in \mathcal{G}$).
Moreover, any basic normal LE-logic and its fully residuated expansion  are both complete w.r.t.~the same class of complete normal LEs, and the latter is a conservative expansion of the former (cf.~\cite[Theorem 12]{ucaaptt},\cite[Section 1.3]{conradie2019algorithmic}).
%
Henceforth, every LE is assumed to be normal, so the adjective `normal' is typically dropped.

Each language $\mathcal{L}_\mathrm{LE}$ is interpreted in the appropriate class of LEs by considering the unique homomorphic extensions of assignments of proposition variables. For every LE $\bba$, the symbol $\vdash$ in  sequents $\varphi\vdash \psi$ is interpreted as the lattice order $\leq$. That is, sequent $\varphi\vdash\psi$ is valid in $\bba$ if $h(\varphi)\leq h(\psi)$ for every homomorphism $h$ from the $\mathcal{L}_\mathrm{LE}$-algebra of formulas over $\mathsf{AtProp}$ to $\bba$. The notation $\mathbb{LE}\models\varphi\vdash\psi$ indicates that $\varphi\vdash\psi$ is valid in every LE. Then it is easy to verify by inspecting the rules that the minimal LE-logic $\mathbf{L}_\mathrm{LE}$ 
 is sound w.r.t.~its corresponding class of algebras $\mathbb{LE}$. Moreover,  a routine Lindenbaum-Tarski construction shows that $\mathbf{L}_\mathrm{LE}$  is also complete w.r.t.~$\mathbb{LE}$-algebras, i.e.\ that any sequent $\varphi\vdash\psi$ is provable in $\mathbf{L}_\mathrm{LE}$ iff $\mathbb{LE}\models\varphi\vdash\psi$. 

\subsection{Inductive LE-inequalities}
\label{ssec: inductive}

In this section we  recall the definitions of inductive LE-inequalities introduced in \cite{conradie2019algorithmic} and their corresponding `analytic' restrictions introduced in \cite{ucaaptt} in the distributive setting and then generalized to the setting of  LEs of arbitrary signatures in \cite{LEframes}.  Inequalities in  these classes 
are all  canonical  and elementary \cite[Theorems 7.1 and 6.1]{conradie2019algorithmic}.  



\begin{definition}
	\label{def: signed gen tree}
	The \emph{positive} (resp.~\emph{negative}) {\em generation tree} of any $\mathcal{L}_\mathrm{LE}$-term $s$ is defined by labelling the root node of the generation tree (i.e.~syntax tree) of $s$ with the sign $+$ (resp.~$-$), and then propagating the labelling on each remaining node as follows:
	\begin{itemize}
		\item For any node labelled with $ \lor$ or $\land$, assign the same sign to its children nodes.
		\item For any node labelled with $h\in \mathcal{F}\cup \mathcal{G}$ of arity $n_h\geq 1$, and for any $1\leq i\leq n_h$, assign the same (resp.~opposite) sign to its $i$th child node if $\varepsilon_{h}(i) = 1$ (resp.~$\varepsilon_{h} (i)= \partial$).
	\end{itemize}
	Nodes in signed generation trees are \emph{positive} (resp.~\emph{negative}) if signed $+$ (resp.~$-$).
\end{definition}

Signed generation trees will  mostly be used in the context of term inequalities $s\leq t$. In this context, we will typically consider the positive generation tree $+s$ for the left-hand side and the negative one $-t$ for the right-hand side. \footnote{\label{footnote: precedent succedent} In the context of sequents $s\vdash t$, signed generation trees $+s$ and $-t$ can also be used to specify when subformulas of $s$ (resp.~$t$) occur in precedent or succedent position. Specifically, a given occurrence of  formula $\gamma$ is in {\em precedent} (resp.~{\em succedent}) position in $s\vdash t$ iff $+\gamma$ is a subtree of $+s$ or of $-t$ (notation: $+\gamma\prec +s$ or $+\gamma \prec -t$) (resp.~$-\gamma$ is a subtree of $+s$ or of $-t$, notation: $-\gamma\prec +s$ or $-\gamma \prec -t$).}

For any term $s(p_1,\ldots p_n)$, any order-type $\epsilon$ over $n$, and any $1 \leq i \leq n$, an \emph{$\epsilon$-critical node} in a signed generation tree of $s$ is a leaf node $+p_i$ if $\epsilon_i = 1$, and a leaf node $-p_i$ if $\epsilon_i = \partial$. An $\epsilon$-{\em critical branch} in the tree is a branch the leaf of which is an $\epsilon$-critical node. 
For every term $s(p_1,\ldots p_n)$ and every order-type $\epsilon$, we say that $+s$ (resp.~$-s$) {\em agrees with} $\epsilon$, and write $\epsilon(+s)$ (resp.~$\epsilon(-s)$), if every leaf in the signed generation tree of $+s$ (resp.~$-s$) is $\epsilon$-critical.
 We also write $+s'\prec \ast s$ (resp.~$-s'\prec \ast s$) to indicate that the subterm $s'$ inherits the positive (resp.~negative) sign from the signed generation tree $\ast s$. Finally, we  write $\epsilon(\gamma) \prec \ast s$ (resp.~$\epsilon^\partial(\gamma) \prec \ast s$) to indicate that the signed subtree $\gamma$, with the sign inherited from $\ast s$, agrees with $\epsilon$ (resp.~with $\epsilon^\partial$).

\begin{definition}
	\label{def:good:branch}
	Non-leaf nodes in signed generation trees are called \emph{$\Delta$-adjoints}, \emph{syntactically left residuals (SLR)}, \emph{syntactically right residuals (SRR)}, and \emph{syntactically right adjoints (SRA)}, according to the specification given in Table \ref{Join:and:Meet:Friendly:Table}.
	Nodes that are either classified as  $\Delta$-adjoints or SLR are collectively referred to as {\em Skeleton-nodes}, while SRA- and SRR-nodes are referred to as {\em PIA-nodes}.
	A branch in a signed generation tree $\ast s$, with $\ast \in \{+, - \}$, is called a \emph{good branch} if it is the concatenation of two paths $P_1$ and $P_2$, one of which may possibly be of length $0$, such that $P_1$ is a path from the leaf consisting (apart from variable nodes) only of PIA-nodes, and $P_2$ consists (apart from variable nodes) only of Skeleton-nodes. 
	A good branch is \emph{Skeleton} if the length of $P_1$ is $0$, 
	and  is {\em SLR}, or {\em definite}, if  $P_2$ only contains SLR nodes.

	\begin{table}[h] { \footnotesize
		\begin{center}
			\bgroup
			\def\arraystretch{1.2}
			\begin{tabular}{| c | c |}
				\hline
				Skeleton &PIA\\
				\hline
				$\Delta$-adjoints & Syntactically Right Adjoint (SRA) \\
				\begin{tabular}{ c c c c c c}
					$+$ &$\vee$ &\\
					$-$ &$\wedge$ \\
				\end{tabular}
				&
				\begin{tabular}{c c c c }
					$+$ &$\wedge$ &$g$ & with $n_g = 1$ \\
					$-$ &$\vee$ &$f$ & with $n_f = 1$ \\

				\end{tabular}
				\\ \hline
				Syntactically Left Residual (SLR) &Syntactically Right Residual (SRR)\\
				\begin{tabular}{c c c c }
					$+$ &  &$f$ & with $n_f \geq 1$\\
					$-$ &  &$g$ & with $n_g \geq 1$ \\
				\end{tabular}
				&\begin{tabular}{c c c c}
					$+$ & &$g$ & with $n_g \geq 2$\\
					$-$ &  &$f$ & with $n_f \geq 2$\\
				\end{tabular}
				\\
				\hline
			\end{tabular}
			\egroup
		\end{center} }
		\caption{Skeleton and PIA nodes for $\mathrm{LE}$-languages.}\label{Join:and:Meet:Friendly:Table}
		\vspace{-1em}
	\end{table}
\end{definition}

\begin{definition}\label{Inducive:Ineq:Def}
	For any order-type $\epsilon$ and any irreflexive and transitive relation (i.e.~strict partial order) $\Omega$ on $p_1,\ldots p_n$, the signed generation tree $*s$ $(* \in \{-, + \})$ of a term $s(p_1,\ldots p_n)$ is \emph{$(\Omega, \epsilon)$-inductive} if, for all $1 \leq i \leq n$,
	\begin{enumerate}
		\item every $\epsilon$-critical branch with leaf $p_i$ is good (cf.~Definition \ref{def:good:branch});
		\item every $m$-ary SRR-node occurring in the critical branch is of the form 
        \[\circledast(\gamma_1,\dots,\gamma_{j-1},\beta,\gamma_{j+1}\ldots,\gamma_m),\]
        \noindent where for any $\ell\in\{1,\ldots,m\}\setminus \{j\}$, 
		\begin{enumerate}
			\item $\epsilon^\partial(\gamma_\ell) \prec \ast s$ (cf.~discussion before Definition \ref{def:good:branch}), and
			%
			\item $p_k <_{\Omega} p_i$ for every $p_k$ occurring in $\gamma_\ell$ and for every $1\leq k\leq n$.
		\end{enumerate}
	\end{enumerate}
	
	We refer to $<_{\Omega}$ as the \emph{dependency order} on the variables. An inequality $s \leq t$ is \emph{$(\Omega, \epsilon)$-inductive} if the signed generation trees $+s$ and $-t$ are $(\Omega, \epsilon)$-inductive. An inequality $s \leq t$ is \emph{inductive} if it is $(\Omega, \epsilon)$-inductive for some $\Omega$ and $\epsilon$.  An inductive inequality is {\em analytic} if every branch is good, and is {\em definite} if no $\Delta$-adjoint nodes (i.e.~$-\wedge$ and $+\vee$) occur in its Skeleton.
\end{definition}

\begin{lemma}
\label{lemma: from inductive to definite sequents} 
\cite[Lemma 8.2]{conradie2019algorithmic} Any inductive inequality is equivalent to a conjunction of definite inductive inequalities.
\end{lemma}	

In what follows,  formulas $\varphi$ such that only Skeleton nodes occur in $+\varphi$ (resp.~$-\varphi$) will be referred to as {\em positive} (resp.~{\em negative}) {\em Skeleton formulas}, and inequalities $\varphi\leq \psi$ such that all nodes of $+\varphi$ and $-\psi$ are Skeleton are referred to as {\em Skeleton inequalities}. Without loss of generality, we will restrict ourselves to definite inductive formulas (cf.~Lemma \ref{lemma: from inductive to definite sequents}).

\subsection{The algorithm ALBA, informally}\label{subseq:informal}
	The  algorithm ALBA is one of the main tools of unified correspondence theory \cite{conradie2012algorithmic, conradie2019algorithmic, CoGhPa14}. In the present subsection, we illustrate it by means of an example. This presentation is based on analogous illustrations in \cite{ucaaptt} \cite{CFPS} and \cite{CGPSZ14}. 
	
	\bigskip
	
	Consider the LE-axiom $\Diamond \Box p \vdash \Box \Diamond\Box p$, which is $\epsilon$-Sahlqvist for $\epsilon(p) = 1$, but it is not analytic, since the branch on the right-hand side of the inequality is not good. Let us recall  that, by Birkhoff's representation theorem, every complete lattice is isomorphic to the concept lattice $\mathbb{P}^+$ of some polarity $\mathbb{P} = (A, X, I)$. Let $\bbA$ be a complete modal lattice such that the operations $\Diamond$ and $\Box$ on $\bbA$ are completely join-preserving and completely meet-preserving, respectively. ALBA starts with the following validity clause of the axiom above on $\bbA$:

	\begin{equation}\label{Church:Rosser}
	\bbA \models \forall p (\Diamond\Box p \leq \Box\Diamond \Box p),
	\end{equation}
	where the logical entailment relation   is interpreted as the ordering relation  $\leq$ on $\bbA$.  Given that we can assume w.l.o.g.~that $\bbA$ is based on a concept lattice, for every  $u\in \bbA$, $u = \bigvee \{\nomj\mid \nomj\leq u\}$ and $u = \bigwedge \{\cnomm\mid  u\leq \cnomm\}$, where concepts $\nomj$ (resp.~$\cnomm$) are the Galois-stable elements of $\mathbb{P}^+$ generated by elements $a\in A$ (resp.~$x\in X$) (cf.~\cite[Proposition 3.1]{conradie2020non}). Following the literature \cite{conradie2012algorithmic}, we  refer to the former variables as {\em nominals}, and  to the latter ones as {\em co-nominals}.  Hence, the condition above can  equivalently be rewritten as follows:
	\begin{equation*}
	\bbA\models \forall p \left(\bigvee\{\nomj\mid \nomj \leq\Diamond \Box p\} \leq \bigwedge\{\cnomm\mid\Box\Diamond \Box p\leq \cnomm\}\right),
	\end{equation*}

    \noindent and since the operations $\Diamond$ and $\Box$ on $\bbA$ are completely join-preserving and completely meet-preserving, the condition above can  equivalently be rewritten as follows:
	\begin{equation*}
	\bbA\models \forall p \left(\bigvee\{\Diamond\nomj\mid \nomj \leq \Box p\} \leq \bigwedge\{\Box\cnomm\mid\Diamond \Box p\leq \cnomm\}\right).
	\end{equation*}
	The condition above holds if and only if every element in the join on the left-hand side is less than or equal to every element in the meet on the right-hand side; thus, condition \eqref{Church:Rosser} above  can be rewritten as:
	\begin{equation} 
	\bbA\models \forall p \forall \nomj \forall \cnomm [(\nomj \leq \Box p \,\,\,\& \,\,\, \Diamond \Box p\leq \cnomm) \Rightarrow \Diamond \nomj \leq \Box \cnomm ].
	\end{equation}
	Since $\Box$ is completely meet-preserving, by well known order-theoretic facts, the left adjoint of $\Box$ exists, which we denote  $\Diamondblack$. Hence, the condition above can equivalently be rewritten as
	\begin{equation}\label{Eq:First:Approx}
	\bbA\models \forall p \forall \nomj \forall \cnomm [(\Diamondblack\nomj \leq  p \,\,\,\& \,\,\, \Diamond \Box p\leq \cnomm) \Rightarrow \Diamond \nomj \leq \Box \cnomm ].
	\end{equation}
	At this point we are in a position to eliminate the variable $p$ and equivalently rewrite the previous condition as follows:
	\begin{equation}
	\label{After:Ack:Eq}
	\bbA\models \forall \nomj \forall \cnomm [ \Diamond \Box \Diamondblack \nomj \leq \cnomm \Rightarrow \Diamond \nomj \leq \Box \cnomm ].
	\end{equation}
	Let us justify this equivalence: for the direction from top to bottom, fix an interpretation $V$ of the variables $\nomj$ and $\cnomm$ such that  $ \Diamond \Box \Diamondblack \nomj \leq \cnomm$. To prove that $\Diamond \nomi\leq \Box \cnomm$ holds under  $V$,  consider the variant $V^\ast$ of $V$ such that $V^\ast(p) = \Diamondblack \nomj$. Then  $V^{\ast}$ satisfies the antecedent of \eqref{Eq:First:Approx} under $V$; hence $\Diamond\nomi\leq \Box\cnomm$ holds under $V$. Conversely,  fix an interpretation $V$ of $p$,  $\nomj$ and $\cnomm$ such that
	$\Diamondblack \nomj \leq p$ and $\Diamond \Box  p \leq \cnomm$. Then, by monotonicity, $\Diamond\Box \Diamondblack \nomj\leq \Diamond\Box p\leq \cnomm$  i.e.~the antecedent of (\ref{After:Ack:Eq}) holds under $V$, and hence so does $\Diamond \nomj\leq \Box \cnomm$, as required. This is an instance of the  {\em Ackermann lemma} (\cite{Ack35}, see also \cite{Conradie:et:al:SQEMAI}).

	Taking stock,  we have equivalently transformed (\ref{Church:Rosser}) into (\ref{After:Ack:Eq}), which is a condition in which all propositional variables  have been eliminated, and all remaining variables are nominals or conominals.\footnote{Nominals and conominals range over elements of the
complex algebra $\bbA$ which correspond to first-order definable elements of the polarity
from which the underlying lattice of $\mathbb{A}$ arises. Hence, condition (\ref{After:Ack:Eq}) can be translated into a first order condition in the language of polarities (suitably augmented with relations $R_{\Box}$ and $R_\Diamond$), which will be the first order correspondent of the initial axiom on polarity-based frames. However, notice that the only properties that are used in this computation is that the elements of $\bbA$ over which nominals and conominals range are complete join-generators and meet-generators respectively. Hence, when applying ALBA for the purpose of computing rules, we can assume w.l.o.g.~that nominals and conominals range over arbitrary elements of $\bbA$. This observation immediately implies the soundness of the inception rules obtained from the (polarity safe) outputs of ALBA, since rules are equivalent to ALBA outputs.}

    Notice that the attempt at transforming condition (\ref{After:Ack:Eq}) into a structural rule as follows:
    \begin{center}
        \AXC{$\Diamond \Box \Diamondblack \nomj \fCenter \cnomm$}
        \UIC{$\Diamond \nomj \fCenter \Box \cnomm$}
        \DP
    \end{center}
    does not work, because  the $\Box$ operator in the premise occurs in precedent position. Not by chance, this is also the reason why the (Sahlqvist) axiom $\Diamond \Box p \vdash \Box \Diamond\Box p$ is not analytic inductive. However, still using the same  (ALBA) manipulations illustrated above,  condition (\ref{After:Ack:Eq}) can be further equivalently  rewritten so as to `unravel' the  nesting of operations, as follows: firstly, because nominals completely join-generate $\bbA$, we have:
    
\begin{equation}
	\bbA\models \forall \nomj \forall \cnomm \left( \bigvee \{\nomi\mid \nomi\leq \Diamond \Box \Diamondblack \nomj\} \leq \cnomm \Rightarrow \Diamond \nomj \leq \Box \cnomm \right).
	\end{equation}

Then, recalling that $\Diamond$ is completely join preserving,

\begin{equation}
	\bbA\models \forall \nomj \forall \cnomm \left( \bigvee \{\Diamond\nomi\mid \nomi\leq  \Box \Diamondblack \nomj\} \leq \cnomm \Rightarrow \Diamond \nomj \leq \Box \cnomm \right);
	\end{equation}
hence, by definition of supremum,
\begin{equation}
	\bbA\models \forall \nomj \forall \cnomm \left( \forall \nomi (\nomi\leq  \Box \Diamondblack \nomj\Rightarrow  \Diamond\nomi\leq \cnomm) \Rightarrow \Diamond \nomj \leq \Box \cnomm \right),
	\end{equation}

    \noindent and again, recalling that co-nominals completely meet-generate $\bbA$ and $\Box$ is completely meet-preserving,
    \begin{equation}
	\label{eq: polarity-safe}
	\bbA\models \forall \nomj \forall \cnomm \left( \forall \nomi (\forall \cnomn(\Diamondblack \nomj\leq \cnomn\Rightarrow \nomi\leq \Box \cnomn)\Rightarrow  \Diamond\nomi\leq \cnomm) \Rightarrow \Diamond \nomj \leq \Box \cnomm \right).
	\end{equation}
    Notice that all connectives in the condition above occur on the side of the inequalities which would allow them to be translated into structural connectives. This syntactic shape is captured by the following
    \begin{definition} 
    \label{def:polarity_safe}
    An ALBA-output is \emph{polarity-safe} when every inequality $\varphi \leq \psi$ in it is Skeleton (cf.~discussion after Definition \ref{Inducive:Ineq:Def}), and 
each nominal (resp.~conominal) in $+\varphi$ and $-\psi$ occurs positively (resp.~negatively).

\end{definition}
\begin{prop} \label{prop:polarity_safe_output}
(cf.~\cite[
Lemma 4.13]{palmigiano2024unifiedinversecorrespondencelelogics})  The ALBA-output of every inductive LE-inequality can be equivalently transformed, via ALBA-rules, into a polarity-safe shape.
\end{prop}
\subsection{Proper display calculi for basic normal LE-logics} \label{ssec:display}

In this section, we recall the definition of
the proper display calculus $\mathrm{D.LE}$ for the basic normal $\mathcal{L}$-logic for a fixed but arbitrary LE-signature $\mathcal{L} = \mathcal{L}(\mathcal{F}, \mathcal{G})$  (cf.~Section \ref{ssec: LE-logics}). 
Our presentation is a more streamlined version of the one introduced in \cite{ucaaptt} for DLE-logics and then straightforwardly generalized to LE-logics in \cite{ChnGrePalTzi21}. 
 
 Let $S_{\mathcal{F}}: = \{\FH \mid f\in \mathcal{F}^*\}$ and $S_{\mathcal{G}}: = \{\GC \mid g\in \mathcal{G}^*\}$ be the sets of structural connectives indexed by  $\mathcal{F}^*$ and $ \mathcal{G}^*$ respectively. Each structural connective comes with an arity and an order-type which coincides with those of its associated operational connective in $ \mathcal{F}^*$ and $\mathcal{G}^*$.
The calculus $\mathrm{D.LE}$ manipulates sequents $\Pi \fCenter \Sigma$, where 
$\Pi$ and $\Sigma$ are structures of two sorts which are built from formulas and are defined by the following simultaneous recursion:

\begin{center}
$\mathsf{Str}_\mathcal{F} \ni \Pi  \cceq A \mid \AATOP \mid \FH\, (\OXI) \ \ \ \ \ \ \ \ \ \ \mathsf{Str}_\mathcal{G} \ni \Sigma \cceq A \mid \ABOT \mid \GC\, (\OXI)$
\end{center}

\noindent where $A$ is an $\mathcal{L}_{\mathrm{LE}}$-formula, $\FH \in S_{\mathcal{F}}$, $\GC \in S_{\mathcal{G}}$, and $\OXI\in \mathsf{Str}_\mathcal{F}^{\epsilon_f}$ (resp.~$\OXI\in \mathsf{Str}_\mathcal{G}^{\epsilon_g}$), 
 where $\mathsf{Str}_\mathcal{F}^{\epsilon_f} = \mathsf{Str}_\mathcal{F}^{\epsilon_{f,1}}\times\cdots\times \mathsf{Str}_\mathcal{F}^{\epsilon_{f,n_f}}$ where $\mathsf{Str}_\mathcal{F}^\partial=\mathsf{Str}_\mathcal{G}$, and dually for $\mathsf{Str}_\mathcal{G}^{\epsilon_g}$. In what follows, for every $o \in S_{\mathcal{F}} \cup S_{\mathcal{G}}$ we use $o(\OXI)[\Pi]_i$ (resp.~$o(\OXI)[\Sigma]_i$) to indicate that the structure $\Pi$ (resp.~$\Sigma$) occurs in the $i$-th coordinate of the vector $\OXI$.

If $S_1$ and $S_2$ are sequents, the  double-horizontal-line notation 
{\fns
\AXC{$S_1$}
\LL{$r$}
\doubleLine
\UIC{$S_2$}
\DP}
is an abbreviation for the rules {\fns  
\AXC{$S_1$}
\LL{$r$}
\UIC{$S_2$}
\DP} 
and 
{\fns 
\AXC{$S_2$}
\RL{$r^{-1}$}
\UIC{$S_1$}
\DP}.

\begin{itemize}
\item Identity and Cut rules:
\end{itemize}

\begin{center}
\begin{tabular}{c}
\AXC{\phantom{$\Gamma \fCenter \varphi$}}
\LL{\fns Id}
\UI$p \fCenter p$
\DP
 \ \ \ 
\AX$\Pi \fCenter A$
\AX$A \fCenter \Sigma$
\RL{\fns Cut}
\BI$\Pi \fCenter \Sigma$
\DP
 \\
\end{tabular}
\end{center}

\begin{itemize}		
\item Display rules for $f\in \mathcal{F}$ and $g\in \mathcal{G}$: for any $1\leq k \leq n_f$ and $1\leq \ell \leq n_g$,\footnote{The notation $\FH \dashv \FCS_k$ (resp.~$\GHF_{\ell} \dashv \GC$) indicates that $\FH$ and $\FCS_k$ (resp.~$\GHF_{\ell}$ and $\GC$) are in a {\em residuated pair} and $\FCS_k$ (resp.~$\GHF_{\ell}$) is the right residual (resp.~left residual) of $\FH$ (resp.~$\GC$) in the $k$-th coordinate (resp.~$\ell$-th coordinate). The notation $(\GC, \GCF_{\ell})$ (resp.~$(\FH, \FHS_k)$) indicates that $\GC$ and $\GCF_{\ell}$ (resp.~$\FH$ and $\FHS_k$) are in a {\em Galois connection} (resp.~{\em dual Galois connection}) and $\GCF_{\ell}$ (resp.~$\FHS_k$) is the right residual (resp.~left residual) of $\GC$ (resp.~$\FH$) in the $\ell$-th coordinate (resp.~$k$-th coordinate).} 
\end{itemize}

If $\varepsilon_{f,k} = 1$ and $\varepsilon_{g,\ell} = 1$

\begin{center}
\begin{tabular}{cc}
 \\
\AXC{$\FH\, (\OXI)[\Pi]_k \fCenter \Sigma$}
\doubleLine
\LL{\fns $\FH \dashv \FCS_k$}
\UIC{$\Pi \fCenter \FCS_k\, (\OXI)[\Sigma]_k$}
\DP
 \ 
\AXC{$\Pi \fCenter \GC\, (\OXI)[\Sigma]_{\ell}$}
\doubleLine
\RL{\fns $\GHF_{\ell} \dashv \GC$}
\UIC{$\GHF_{\ell}\, (\OXI)[\Pi]_{\ell} \fCenter \Sigma$}
\DP 
\end{tabular}
\end{center}

If $\varepsilon_{f,k} = \partial$ and $\varepsilon_{g,\ell} = \partial$

\begin{center}
\begin{tabular}{c}
\AX$\FH\, (\OXI)[\Sigma]_k \fCenter \Sigma'$
\doubleLine
\LL{\fns $(\FH, \FHS_k)$}
\UI$\FHS_k\, (\OXI)[\Sigma']_k \fCenter \Sigma$
\DP
 \ 
\AX$\Pi' \fCenter \GC\, (\OXI)[\Pi]_{\ell}$
\doubleLine
\RL{\fns $(\GC, \GCF_{\ell})$}
\UI$ \Pi \fCenter \GCF_j\, (\OXI)[\Pi']_{\ell}$
\DP
\\
\end{tabular}
\end{center}



\begin{itemize}
\item Logical introduction rules for lattice connectives, for $i \in \{1,2\}$:\footnote{The calculus D.LE presented here can be refined for capturing  the basic normal DLE-logics, i.e.~the logics  of normal  lattice expansions based on {\em distributive} lattices. Noticing that, in distributive lattices, $\wedge$ (resp.~$\vee$) has the properties of a binary $\mathcal{F}$-connective (resp.~$\mathcal{G}$-connective) positive in both coordinates,  the introduction rules for $\wedge$ (resp.~$\vee$) given above can be replaced by the corresponding instantiations of the introduction rules for binary $\mathcal{F}$-connectives (resp.~$\mathcal{G}$-connectives); hence, structural connectives and their residuals will be added to the language of the calculus; finally, the standard structural rules of weakening, contraction, associativity and exchange will be added to the calculus.}
\end{itemize}

\begin{center}
\begin{tabular}{c}
\AX$A_i \fCenter \Sigma$
\LL{\fns $\aand_{Li}$}
\UI$A_1 \aand A_2 \fCenter \Sigma$
\DP
 \ 
\AX$\Pi \fCenter A$
\AX$\Pi \fCenter B$
\RL{\fns $\aand_R$}
\BI$\Pi \fCenter A \,\aand\, B$
\DP
\
\AXC{ \ }
\RL{\fns $\aatop$}
\UI$\Pi\fCenter \aatop$
\DP
 \\
    
 \\
 \
 \AXC{ \ }
\LL{\fns $\abot$}
\UI$\abot \fCenter \Sigma$
\DP
 \ 
\AX$A \fCenter \Sigma$
\AX$B \fCenter \Sigma$
\LL{\fns $\aor_L$}
\BI$A \,\aor\, B \fCenter \Sigma$
\DP
 \ 
\AX$\Pi \fCenter A_i$
\RL{\fns $\aor_{Ri}$}
\UI$\Pi \fCenter A_1 \aor A_2$
\DP
 \\
\end{tabular}
\end{center}

\begin{itemize}
\item Logical introduction rules for $f\in\mathcal{F}$ and $g\in\mathcal{G}$: 
\end{itemize}
				\begin{center}
					\begin{tabular}{c c c c}
						\bottomAlignProof
						\AX$\FH\, (\overline{A}) \fCenter \Sigma$
						\LL{\fns$f_L$}
						\UI$f(\overline{A}) \fCenter \Sigma$
						\DP
						&
						\bottomAlignProof
						\AX$\Pi \fCenter \GC\, (\overline{A})$
						\RL{\fns$g_R$}
						\UI$\Pi \fCenter g(\overline{A})$
						\DP
                        &
                        \bottomAlignProof
						\AxiomC{$\Big(\Xi_i \fCenter^{\!\!\epsilon_{f,i}}\, A_i \Big)_{1\leq i\leq n_f}$}
						\RL{\fns$f_R$}
						\UI$\FH\, (\OXI)\fCenter f(\overline{A})$
						\DP
                        &
                        \bottomAlignProof
						\AxiomC{$\Big(A_i \fCenter^{\!\!\epsilon_{{g,i}}}\; \Xi_i \Big)_{1\leq i\leq n_g}$}
						\LL{\fns$g_L$}
						\UI$g(\overline{A}) \fCenter \GC\, (\OXI)$
						\DP
					\end{tabular}
                \end{center}

		where $\OXI = (\Xi_1,\ldots, \Xi_{n_f})$ in $f_R$ and $\OXI = (\Xi_1,\ldots, \Xi_{n_g})$ in $g_L$.	
                In particular, if $f$ and $g$ are $0$-ary (i.e.~they are constants), the rules $f_R$ and $g_L$ above reduce to the axioms ($0$-ary rules) $\FH \fCenter f$ and $g \fCenter \GC$.

	The calculus $\mathrm{D.LE}$ is sound w.r.t.~the class of complete $\mathcal{L}$-algebras (cf.~\cite[Theorem 12]{ucaaptt} and \cite[Theorem 2.8]{LEframes}).
	 Moreover, D.LE is a proper display calculus (cf.~\cite[Theorem 26]{ucaaptt} and \cite[Appendix A]{LEframes}), and hence cut elimination holds for it as a consequence of a Belnap-style cut elimination metatheorem (cf.~\cite[Section 2.2 and Appendix A]{ucaaptt} and \cite[Theorem A.7]{LEframes}).

In the presentation of the language of D.LE, $\Sigma_i$, $\Pi_i$, and $\Xi_i$ are metavariables for structures. To formally present analytic structural rules, we need to make use of metastructures, i.e., structures that are constructed by structural metavariables. Let $\mathsf{MVar} = \mathsf{MVar}_{\mathcal{F}}\uplus \mathsf{MVar}_{\mathcal{G}}$ be the denumerable set of {\em metavariables} of sorts $\Pi_1, \Pi_2, \ldots\in \mathsf{MVar}_{\mathcal{F}}$ and $\Sigma_1, \Sigma_2, \ldots \in \mathsf{MVar}_{\mathcal{G}}$.
The sets $\mathsf{MStr}_\mathcal{F}$ and $\mathsf{MStr}_\mathcal{G}$ of the $\mathcal{F}$- and $\mathcal{G}$-{\em metastructures} are defined by simultaneous induction as follows:

\begin{center}
\begin{tabular}{c}
$\mathsf{MStr}_\mathcal{F} \ni U \ \cceq \  \Pi \mid \FH(\overline{U})$\quad \quad \quad
$\mathsf{MStr}_\mathcal{G} \ni V \ \cceq \  \Sigma \mid \GC(\overline{V})$
\end{tabular}
\end{center}
where $\Pi\in \mathsf{MVar}_{\mathcal{F}}$, $\Sigma\in \mathsf{MVar}_{\mathcal{G}}$, $\FH \in \mathcal{F}^\ast$ and $\GC \in \mathcal{G}^\ast$ and $\overline{U} \in \mathsf{MStr}_\mathcal{F}^{\epsilon_f}$,  and $\overline{V} \in \mathsf{MStr}_\mathcal{G}^{\epsilon_g}$, and for any order-type $\epsilon$ on $n$, we let $\mathsf{MStr}_\mathcal{F}^{\epsilon} \coloneqq \prod_{i = 1}^{n}\mathsf{MStr}_\mathcal{F}^{\epsilon_i}$ and $\mathsf{MStr}_\mathcal{G}^{\epsilon} \coloneqq \prod_{i = 1}^{n}\mathsf{MStr}_\mathcal{G}^{\epsilon_i}$, where for all $1 \leq i \leq n$,
\begin{center}
\begin{tabular}{ll}
$\mathsf{MStr}_\mathcal{F}^{\epsilon_i} = \begin{cases} 
\mathsf{MStr}_\mathcal{F} &\mbox{ if } \epsilon_i = 1\\
\mathsf{MStr}_\mathcal{G} &\mbox{ if } \epsilon_i = \partial
\end{cases}\quad$
&
$\mathsf{MStr}_\mathcal{G}^{\epsilon_i} = \begin{cases}
\mathsf{MStr}_\mathcal{G}& \mbox{ if } \epsilon_i = 1,\\
\mathsf{MStr}_\mathcal{F} & \mbox{ if } \epsilon_i = \partial.
\end{cases}$
\end{tabular}
\end{center}

An {\em analytic structural rule} is  a rule of the form 
\begin{center}
\AXC{$(U_1\fCenter V_1)[\overline{\Upsilon}_1]$}
\AXC{$\cdots$}
\AXC{$(U_n\fCenter V_n)[\overline{\Upsilon}_n]$}
\TIC{$(U_0 \fCenter V_0)[\overline{\Upsilon}_0]$}
\DP
\end{center}
where $\overline{\Upsilon}_i$ with $0 \leq i \leq n$ is the set of metavariables occurring in each sequent $U_i \fCenter V_i$, $\overline{\Upsilon}_0 \supseteq \overline{\Upsilon}_1 \cup \ldots \cup \overline{\Upsilon}_n$, and while structural metavariables might occur multiple times in the premises they occur only once in the conclusion. 
An instance of the rule $R$ is obtained from $R$ by uniformly substituting each structural metavariable $\Pi\in \mathsf{MVar}_{\mathcal{F}}$ with an element of $\mathsf{Str}_{\mathcal{F}}$ and every structural metavariable $\Sigma\in \mathsf{MVar}_{\mathcal{G}}$ with an element of $\mathsf{Str}_{\mathcal{G}}$. A calculus contains the analytic rule $R$ if it contains every instance of $R$.

\begin{prop}(cf.~\cite[Proposition 59, Proposition 61]{ucaaptt})
	\label{prop:type5} Every analytic $(\Omega, \varepsilon)$-inductive LE-inequality can be equivalently transformed, via an ALBA-reduction, into a set of analytic structural rules. Conversely, every analytical structural rule is equivalent to an $(\Omega, \varepsilon)$-inductive LE-inequality.
\end{prop}

\section{Beyond analytic inductive axioms} \label{sec:inception}

The purpose of this contribution is to extend the class of axioms which can be captured by structural rules which preserve  cut-elimination  when added to a proper display calculus. Since  analytic structural rules correspond precisely to the class of analytic inductive inequalities (cf.~Definition \ref{Inducive:Ineq:Def}), in order to capture inductive axioms we introduce structural rules with higher-order assumptions. We show that these rules greatly enhance the expressive power of proper display calculi while preserving  cut-elimination via a Belnap-style metatheorem. We refer to proper display calculi augmented with these rules as \emph{inception display calculi}.

\subsection{Inception display calculi} 
Inception display calculi augment analytic structural rules with a certain type of higher-order assumptions,  defined below.

\begin{definition} \label{def:depthrules}
 Let us recursively define  \emph{depth-$n$ inception rules}  and \emph{depth-$n$ contracts} (for any $n\geq 0$).
\begin{itemize}
    \item Depth-$0$ inception rules are the analytic structural rules. Depth-$0$ contracts are just sequents built from structure metavariables.

    \item Case $n>0$. Assuming that  depth-$k$ inception rules and contracts have been defined for every $k < n$, a depth-$n$ contract is an expression of the form $\left[ \Pi\fCenter\Sigma\right]^\mathcal{R}_\mathcal{X}$, where the set of contract-rules $\mathcal{R}$ is a set of inception rules of depth smaller than $n$, and containing at least one inception rule of depth $n-1$, while $\mathcal{X}$ is a set of structure metavariables that should be understood as universally quantified and hence should not be instantiated in concrete applications of the rule. We refer to $\mathcal{R}$ as the set of \emph{contract rules}. In $\left[ \Pi\fCenter\Sigma\right]^\mathcal{R}_\mathcal{X}$, $\Pi,\Sigma$, and every rule in $\mathcal{R}$ may contain structure variables from $\mathcal{X}$, but not necessarily. We say that the contract $\left[ \Pi\fCenter\Sigma\right]^\mathcal{R}_\mathcal{X}$ is \emph{fulfilled} or \emph{holds} if $\Pi \fCenter \Sigma$ (referred to as the {\em aim} of the contract) is derivable using the rules of the base proper display calculus, possibly (but not necessarily) together with those in $\mathcal{R}$.  
    
    A depth-$n$ inception rule is a structural rule whose premises are contracts of depth not greater than $n$, where at least one of them has depth $n-1$. Contracts of positive depth are referred to as \emph{higher-order assumptions}. 
\end{itemize}
An \emph{inception rule} is a depth-$n$ inception rule for some $n \geq 1$.
The general shape of an inception rule is

\begin{center}
    \AXC{$\Pi_1 \fCenter \Sigma_1\ \ \cdots\ \ \Pi_n \fCenter \Sigma_n\ \ \ \ \left[ \Pi_{n+1} \fCenter \Sigma_{n+1}\right]^{\mathcal{R}_1}_{\mathcal{X}_1}\ \ \cdots\ \ \left[ \Pi_{n+m} \fCenter \Sigma_{n+m}\right]^{\mathcal{R}_m}_{\mathcal{X}_m}$}
    \RL{.}
\UIC{$\Pi \fCenter \Sigma$}
\DP

\end{center}

An instance of an inception rule is obtained by

\begin{enumerate}
    \item uniformly substituting each instantiable structural metavariable $\Pi\in \mathsf{MVar}_{\mathcal{F}}$ (resp.~$\Sigma\in \mathsf{MVar}_{\mathcal{G}}$) with an element of $\mathsf{Str}_{\mathcal{F}}$ (resp.~$\mathsf{Str}_{\mathcal{G}}$);
    \item providing, for each higher-order assumption $\left[ \Pi\fCenter\Sigma\right]^\mathcal{R}_\mathcal{X}$ of the rule, a derivation $\pi$ fullfilling the corresponding contract.
\end{enumerate}  
\end{definition}

\begin{example}
    In Section \ref{subseq:informal}, we computed the first-order correspondent of the LE-axiom $\Diamond \Box p \vdash \Box \Diamond\Box p$, obtaining  
    \[\forall \nomj \forall \cnomm ( \Diamond \Box \Diamondblack \nomj \leq \cnomm \Rightarrow \Diamond \nomj \leq \Box \cnomm ).\]
    This sentence is not translatable into an analytic structural rule, because the premise $\Diamond \Box \Diamondblack \nomj \leq \cnomm$ is not translatable into a sequent. However, since the inequality $\Diamond \Box \Diamondblack \nomj \leq \cnomm$ is equivalent to the expression 
    \[\forall \nomi (\forall \cnomn(\Diamondblack \nomj\leq \cnomn\Rightarrow \nomi\leq \Box \cnomn)\Rightarrow  \Diamond\nomi\leq \cnomm),\]
    the premise can be rewritten as the higher-order assumption $[\hat{\Diamond}I \vdash M]^{R}_{\{I\}}$, where $R$ is the contract rule
    \begin{center}
        \AXC{$\hat{\Diamondblack}J \vdash N$}
        \RL{$R$}
        \UIC{$I \vdash \check{\Box}N$}
        \DP
    \end{center}
\end{example}

\begin{notation} \label{notation_dream}
    In applying an inception rule with $\left[ \Pi\fCenter\Sigma\right]^\mathcal{R}_\mathcal{X}$ as a higher-order assumption, in drawing the derivation tree we will sometimes write $\left[ \pi\right]^\mathcal{R}_\mathcal{X}$ in place of $\left[ \Pi\fCenter\Sigma\right]^\mathcal{R}_\mathcal{X}$, where $\pi$ is a derivation of $\Pi\fCenter\Sigma$ witnessing the lawful application of the rule. Under these circumstances, we say that $\pi$ is a witness for the contract $\left[ \Pi\fCenter\Sigma\right]^\mathcal{R}_\mathcal{X}$. We omit subscripts and superscripts when they are clear from the context. In the case just described, we evocatively say that $\pi$ \emph{is in a dream}. 
\end{notation}

    If the higher-order assumption $\left[ \Pi\fCenter\Sigma\right]^\mathcal{R}_\mathcal{X}$ is fulfilled by $\pi$ in the application of an inception rule in a derivation $\pi'$, we consider $\pi$ to be a subderivation of $\pi'$. This implies that a derivation containing an instance of the cut rule inside one of its dreams is not cut-free.

\begin{definition}
    A derivation has \emph{finite dreams} if it contains a finite number of instances of contracts. Since we regard derivations living in dreams as proper parts of the original derivation (cf.~discussion below Notation \ref{notation_dream}), having finite dreams coincides with being well-founded w.r.t.~`inception depth'. 
\end{definition}

A core assumption of the present paper is that  all derivations have finite dreams; moreover, we assume that in each derivation, the structure variables that parametrize contracts are all different.

\subsection{Generating inception rules}

In this subsection, we describe how to translate from the polarity-safe ALBA output of an inductive $\mathrm{LE}$-inequality (cf.~Definition \ref{def:polarity_safe} and Proposition \ref{prop:polarity_safe_output}) to its equivalent inception rule(s). We begin with some preliminary definitions, refining the class of inductive formulas, subdividing them according to their complexity in terms of the alternation between Skeleton and PIA connectives.

\begin{definition} \label{def:branch_depth}
    A branch in a signed generation tree (cf.~Section \ref{ssec: inductive}) $*s$, with $* \in \{+,-\}$, has \emph{depth} $n$ if it is the concatenation of $n$ paths $P_0, \dots, P_n$, such that each $P_i$ is composed of two segments, the first (considering the direction from the root to the leaf) composed only of Skeleton nodes and the second composed only of PIA nodes. When $1<i<n$, both segments of $P_i$ are required to be nonempty, while in $P_0$ only its PIA segment is required to be nonempty. Finally, when $n>0$, in $P_n$ only its Skeleton segment is required to be nonempty.
\end{definition}

\begin{definition} \label{def: depth_n_inductive}
    An \emph{inductive inequality} has \emph{depth} $n$ if it has at least one branch of depth $n$ and every branch has depth not greater than $n$. 
\end{definition}

From Definition \ref{def: depth_n_inductive}, it follows that analytic inductive inequalities have depth $0$.

\begin{example}

Consider the inductive inequality $\Box(\Diamond p \circ p) \circ \Diamond p \leq p$ such that $\{\wdia, \circ\} \subseteq \mathcal{F}$, $\wbox \in \mathcal{G}$, $n_\circ=2$, and $n_\wdia = n_\wbox = \varepsilon_\wdia(1)=\varepsilon_\wbox(1)=\varepsilon_\circ(1)=\varepsilon_\circ(2)=1$. The signed generation tree (cf.~Definition \ref{def: signed gen tree}) of the inequality is the following, where the instances of propositional atoms are marked with subscripts to distinguish the branches.

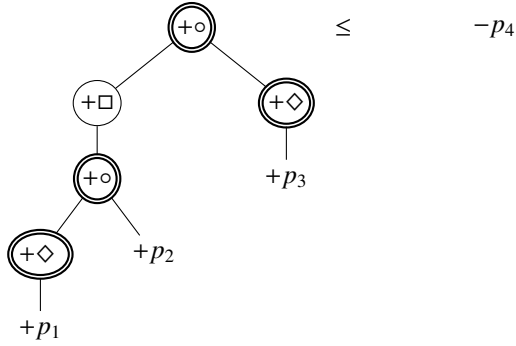
\begin{center}
	\begin{tikzpicture}
		\tikzstyle{level 1}=[level distance=1cm, sibling distance=2.5cm]
		\tikzstyle{level 2}=[level distance=1cm, sibling distance=2.5cm]
		\tikzstyle{level 3}=[level distance=1 cm, sibling distance=1.5cm]
        \tikzstyle{level 4}=[level distance=1 cm, sibling distance=1.5cm]
        \tikzstyle{level 5}=[level distance=1 cm, sibling distance=1.5cm]
        
		\node[Ske] at (-2,0) {$\begin{aligned} +\circ \end{aligned}$}
		child{node[PIA]{$\begin{aligned}+\Box\end{aligned}$}
			child{node[Ske]{$\begin{aligned}+\circ\end{aligned}$}
				child{node[Ske]{
                    {$\begin{aligned}+\Diamond\end{aligned}$}
                }
                child{node{$+p_1$}}
                }
				child{node{$+p_2$}}
			}
		}
        child{node[Ske]{$\begin{aligned}+\Diamond\end{aligned}$}
            child{node{$+p_3$}}
        }
		;
		\node at (0,0) {$\le$}; 
		
		\node at (2,0) {$-p_4$};
	\end{tikzpicture}
\end{center}

In the tree above, the branches leading to $p_1$ and $p_2$ have depth $1$, while the other branches have depth $0$, meaning that $\Box(\Diamond p \circ p) \circ \Diamond p \leq p$ has depth $1$. 
\end{example}

The translation is defined as follows. We start from a generic polarity-safe ALBA output of an inductive inequality of depth $n>0$:
\begin{equation} \label{eq:ALBA_output}
\forall \overline{\nomj} \forall \overline{\cnomm} (\overline{\nomj \leq \alpha}\ \&\ \overline{\beta\leq\cnomm}\ \&\ \bigwith \overline{\Gamma_1} \Rightarrow \varphi \leq \psi),
\end{equation}
where every $\nomj \leq \alpha$ in $\overline{\nomj\leq\alpha}$, every $\beta\leq\cnomm$ in $\overline{\beta\leq\cnomm}$, and $\varphi \leq \psi$ are pure definite Skeleton inequalities. Furthermore, every (co)nominal of $\overline{\nomj}$ and $\overline{\cnomm}$ appears exactly once in $\varphi \leq \psi$, because it comes from a first approximation step (cf.~Example \ref{ex:1}), and does not appear in any $\alpha$ in $\ol{\alpha}$ or $\beta$ in $\ol{\beta}$. Furthermore, all the (co)nominals occurring in $\ol{\nomj \leq \alpha}, \ol{\beta \leq \cnomm}$, $\varphi_1$, and $\psi_1$ are from $\ol{\nomj}$,$\ol{\cnomm}$, with the exception of the (co)nominals in $\ol{\nomj_1}$,$\ol{\cnomm_1}$. Every $\Gamma_1$ in $\overline{\Gamma_1}$ is an expression of one of the forms
\[\forall \overline{\nomj_1}\forall \overline{\cnomm_1}\left( \bigwith_k\left(\forall \overline{\nomh_{1,k}}\forall\overline{\cnomn_{1,k}}(\overline{\alpha_{1,k} \leq \beta_{1,k}}\ \&\ \bigwith \overline{\Gamma_{2,k}}\Rightarrow \gamma_{1,k} \leq\delta_{1,k}) \right)\Rightarrow \nomj \leq \psi_1\right)\]
\[\forall \overline{\nomj_1}\forall \overline{\cnomm_1}\left( \bigwith_k\left(\forall \overline{\nomh_{1,k}}\forall\overline{\cnomn_{1,k}}(\overline{\alpha_{1,k} \leq \beta_{1,k}}\Rightarrow \gamma_{1,k} \leq\delta_{1,k}) \right)\Rightarrow \nomj \leq \psi_1\right)\]
\[\forall \overline{\nomj_1}\forall \overline{\cnomm_1}\left( \bigwith_k\left(\forall \overline{\nomh_{1,k}}\forall\overline{\cnomn_{1,k}}(\overline{\alpha_{1,k} \leq \beta_{1,k}}\ \&\ \bigwith \overline{\Gamma_{2,k}}\Rightarrow \gamma_{1,k} \leq\delta_{1,k}) \right)\Rightarrow \varphi_1 \leq \cnomm\right)\]
\[\forall \overline{\nomj_1}\forall \overline{\cnomm_1}\left( \bigwith_k\left(\forall \overline{\nomh_{1,k}}\forall\overline{\cnomn_{1,k}}(\overline{\alpha_{1,k} \leq \beta_{1,k}}\Rightarrow \gamma_{1,k} \leq\delta_{1,k}) \right)\Rightarrow \varphi_1 \leq \cnomm\right)\]
where, for all $k$, every $\alpha \leq \beta$ in $\overline{\alpha_{1,k}\leq\beta_{1,k}}$, $\gamma_{1,k}\leq\delta_{1,k}$, $\varphi_1\leq\cnomm$, and $\nomj\leq\psi_1$ are pure Skeleton inequalities and every $\Gamma$ in $\overline{\Gamma_{2,k}}$ is a more complex expression. In general, for $i$ in $\{1,\ldots,n-1\}$, every $\Gamma_i$ in $\overline{\Gamma_i}$ has one of the shapes

\[\forall \overline{\nomj_i}\forall \overline{\cnomm_i}\left( \bigwith_k\left(\forall \overline{\nomh_{i,k}}\forall\overline{\cnomn_{i,k}}(\overline{\alpha_{i,k} \leq \beta_{i,k}}\ \&\ \bigwith \overline{\Gamma_{i+1,k}}\Rightarrow \gamma_{i,k} \leq\delta_{i,k}) \right)\Rightarrow \varphi_i \leq \psi_i\right)\ \text{ or}\]
\[\forall \overline{\nomj_i}\forall \overline{\cnomm_i}\left( \bigwith_k\left(\forall \overline{\nomh_{i,k}}\forall\overline{\cnomn_{i,k}}(\overline{\alpha_{i,k} \leq \beta_{i,k}}\Rightarrow \gamma_{i,k} \leq\delta_{i,k}) \right)\Rightarrow \varphi_i \leq \psi_i\right)\]

Finally, every $\Gamma_n$ in $\overline{\Gamma_n}$ is of the form
\[\forall \overline{\nomj_n}\forall \overline{\cnomm_n}\left( \bigwith_k\left(\forall \overline{\nomh_{n,k}}\forall\overline{\cnomn_{n,k}}(\overline{\alpha_{n,k} \leq \beta_{n,k}}\Rightarrow \gamma_{n,k} \leq\delta_{n,k}) \right)\Rightarrow \varphi_n \leq \psi_n\right).\]

Every nominal in $\ol{\nomj_i}$ and conominal in $\ol{\cnomm_i}$ occurs exactly once in $\varphi_i \leq \psi_i$; furthermore, every nominal in $\ol{\nomh_{i,k}}$ and conominal in $\ol{\cnomn_{i,k}}$ occurs exactly once in $\gamma_{i,k} \leq \delta_{i,k}$, because all (co)nominals are introduced via approximation rules \cite[Section 4]{conradie2019algorithmic}.
All the (co)nominals occurring in $\ol{\alpha_{i,k} \leq \beta_{i,k}}$ not from $\ol{\nomj}$,$\ol{\cnomm}$ are from $\ol{\nomh_{i,k}}$,$\ol{\cnomn_{i,k}}$, while all the (co)nominals occurring in $\varphi_{i+1} \leq \psi_{i+1}$ that are not in $\ol{\nomj_{i+1}}$,$\ol{\cnomm_{i+1}}$ are bounded by one of the (co)nominal variables from $\ol{\nomh_{i,k}}$,$\ol{\cnomn_{i,k}}$. Finally, all the (co)nominals occurring in $\gamma_{i,k} \leq \delta_{i,k}$ that are not in $\ol{\nomh_{i,k}}$,$\ol{\cnomn_{i,k}}$ come from $\ol{\nomj_i}$,$\ol{\cnomm_i}$

Having described the general shape of the polarity-safe ALBA output of a depth $n$ definite inductive inequality, it is now easy to extract an equivalent depth-$n$ analytic inception rule from it. The rule is defined as follows:

\begin{center}
\AXC{$\overline{s(\alpha)\fCenter s(\cnomm)}$}
\AXC{$\overline{s(\nomj)\fCenter s(\beta)}$}
\AXC{$\overline{\left[s(\varphi_1) \fCenter s(\psi_1)\right]^{\mathcal{R}_1}_{\{\overline{s(\nomj_1)},\overline{s(\cnomm_1)}\}}}$}
\TIC{$s(\varphi) \fCenter s(\psi)$}
\DP
\end{center}
where $s$ consistently maps every nominal and conominal to a structure variable and every connective to its structural counterpart.\footnote{The structural counterpart of a connective $f \in \mathcal{F}^*$ (resp.~$g \in \mathcal{G}^*$) is the structural connective $\hat{f} \in S_\mathcal{F}$ (resp.~$\check{g} \in S_\mathcal{G}$), as defined in Section \ref{ssec:display}.}  In the rule above, the number of contracts of positive depth is exactly the cardinality of $\overline{\Gamma_1}$, and every $R_{1,k}$ in $\mathcal{R}_1$ is of one of the shapes
\begin{center}
\begin{tabular}{ccc}
\AXC{$\overline{s(\alpha_{1,k})\fCenter s(\beta_{1,k})}$}
\UIC{$s(\gamma_{1,k}) \fCenter s(\delta_{1,k})$}
\DP &
or &
\AXC{$\overline{s(\alpha_{1,k})\fCenter s(\beta_{1,k})}$}
\AXC{$\overline{\left[s(\varphi_{2}) \fCenter s(\psi_2)\right]^{\mathcal{R}_2}_{\{\overline{s(\nomj_2)},\overline{s(\cnomm_2)}\}}}$}
\BIC{$s(\gamma_{1,k}) \fCenter s(\delta_{1,k})$}
\DP
\end{tabular}
\end{center}
In general, for $i$ in $\{1,\ldots,n-1\}$, every $R_{i,k}$ in some $\mathcal{R}_i$ corresponds to an expression $\Gamma_{i,k}$ in $\overline{\Gamma_{i,k}}$ and has one of the shapes
\begin{center}
\begin{tabular}{ccc}
\AXC{$\overline{s(\alpha_{i,k})\fCenter s(\beta_{i,k})}$}
\UIC{$s(\gamma_{i,k}) \fCenter s(\delta_{i,k})$}
\DP &
 or &
\AXC{$\overline{s(\alpha_{i,k})\fCenter s(\beta_{i,k})}$}
\AXC{$\overline{\left[s(\varphi_{i+1}) \fCenter s(\psi_{i+1})\right]^{\mathcal{R}_{i+1}}_{\{\overline{s(\nomj_{i+1})},\overline{s(\cnomm_{i+1})}\}}}$}
\BIC{$s(\gamma_{i,k}) \fCenter s(\delta_{i,k})$}
\DP
\end{tabular}
\end{center}
Finally, every $R_{n,k}$ in a $\mathcal{R}_n$ corresponds to a $\Gamma_{n,k}$ in $\overline{\Gamma_{n,k}}$ and has the shape
\begin{center}
\AXC{$\overline{s(\alpha_{n,k})\fCenter s(\beta_{n,k})}$}
\UIC{$s(\gamma_{n,k}) \fCenter s(\delta_{n,k})$}
\DP
\end{center}

We conclude with a concrete example.

\begin{example}
    Consider the depth-$1$ inductive inequality $(\Box p\ \circ\ \Box{\lhd} p)\circ\Diamond q \leq \ {\lhd} p\star \Diamond \Box q$ in a signature where $\{\Diamond, \circ, \lhd\} \subseteq \mathcal{F}$, $\{\Box,\star\}\subseteq\mathcal{G}$, such that $n_\Diamond = n_\lhd=n_\Box = 1$, $n_\circ = n_\star = 2$ and all the operators are monotone in every coordinate, with the exception of $\lhd$, which is antitone. The residuals of $\Box$ and $\lhd$ are denoted with $\Diamondblack$ and $\blacktriangleleft$, respectively. Following the steps outlined above, we observe that its polarity-safe ALBA output is
    \[\forall \nomi_1 \forall \nomi_2 \forall \nomi_3 \forall \cnomm_1 \forall \cnomm_2 (\Gamma\ \&\ \Gamma' \Rightarrow (\nomi_1\ \circ\ \nomi_2)\circ\Diamond \nomi_3 \leq \ \cnomm_1\star \cnomm_2),\]
    where $\Gamma$ and $\Gamma'$ are the expressions
    \[\forall \cnomn(\forall \cnoml(\Diamondblack\nomi_1\leq\cnoml\ \&\ {\blacktriangleleft}\cnomm_1\leq\cnoml \Rightarrow \lhd\cnoml \leq \cnomn) \Rightarrow \nomi_2\leq \cnomm_2)\ \text{ and}\]
    \[\forall \nomj(\forall \cnomo(\nomi_3\leq\cnomo\Rightarrow\nomj\leq\Box\cnomo)\Rightarrow\Diamond\nomj\leq\cnomm_2)\]
    respectively. The corresponding depth-1 inception rule is
    \begin{center}
\AXC{$\left[s(\nomi_2) \vdash s(\Box \cnomn)\right]^{\mathcal{R}}_{\{s(\cnomn)\}}$}
\AXC{$\left[s(\Diamond\nomj) \vdash s(\cnomm_2)\right]^{\mathcal{R}'}_{\{s(\nomj)\}}$}
\BIC{$s((\nomi_1\ \circ\ \nomi_2)\circ\Diamond \nomi_3) \vdash s(\ \cnomm_1\star \cnomm_2)$}
\DP
\end{center}
    where $\mathcal{R}$ and $\mathcal{R}'$ are singletons containing the rules

\begin{center}
\begin{tabular}{ccc}
\AXC{$s(\Diamondblack\nomi_1) \vdash s(\cnoml)$}
\AXC{$s({\blacktriangleleft}\cnomm_1) \vdash s(\cnoml)$}
\BIC{$s(\lhd\cnoml) \vdash s(\cnomn)$}
\DP &
and &
\AXC{$s(\nomi_3) \vdash s( \cnomo)$}
\UIC{$s(\nomj) \vdash s(\Box \cnomo)$}
\DP
\end{tabular}
\end{center}
    respectively. Finally, if $s(\nomi_k)=X_k$, $s(\cnomm_{k'})= Y_{k'}$, $s(\cnomn)=Z_1$, $s(\nomj)=Z_2$, $s(\cnoml)=W$, and $s(\cnomo)=S$, after we translate everything we get

\begin{center}
\AXC{$\left[X_2 \vdash \check{\Box} Z_1\right]^{\mathcal{R}}_{\{Z_1\}}$}
\AXC{$\left[\hat{\Diamond}Z_2 \vdash Y_2\right]^{\mathcal{R}'}_{\{Z_2\}}$}
\BIC{$(X_1\ \hat{\circ}\ X_2)\ \hat{\circ}\ \hat{\Diamond} X_3 \vdash  Y_1\ \check{\star}\ Y_2$}
\DP
\end{center}
    where $\mathcal{R}$ and $\mathcal{R}'$ are singletons containing the rules

    \begin{center}
\begin{tabular}{ccc}
\AXC{$\hat{\Diamondblack}X_1 \vdash W$}
\AXC{$\hat{\blacktriangleleft}Y_1 \vdash W$}
\BIC{$\hat{\lhd}W \vdash Z_1$}
\DP &
and &
\AXC{$X_3 \vdash S$}
\UIC{$Z_2 \vdash \check{\Box} S$}
\DP
\end{tabular}
\end{center}
respectively.
\end{example}

\subsection{Examples: ALBA-generated inception rules} \label{ssec:examples}
The prime class of inception rules with finite dreams is generated by running the algorithm ALBA \cite{conradie2019algorithmic} on inductive LE-inequalities until a polarity-safe shape has been reached, and then translating the polarity-safe ALBA output. In what follows, we illustrate this procedure by way of examples.

\begin{remark}\label{footn:polarity-safe}
It is always possible to reach a polarity-safe output starting from an inductive inequality. We first eliminate the propositional variables (cf.~steps 1 to 3 in Example \ref{ex:1}), then we decompose further any non-Skeleton formula via approximation rules \cite[Section 4]{conradie2019algorithmic} until only Skeleton terms remain (cf.~step 3 to 4 in Example \ref{ex:1}), and finally, if needed, we adjust the polarity of nominals and conominals by applying more approximation rules (cf.~step 4 to 5 in Example \ref{ex:1}).
 Let us refer to  the nominals and conominals introduced in the first approximation step  as the \emph{extractors}. Notice that, as illustrated by the example below, after reaching the polarity-safe shape, the extractors occurring on the left-hand side of the main quasi-inequality (i.e.~those marked in red) always occur within inequalities {\em negative position}, i.e.~inequalities which are nested on the antecedent side of an odd number of quasi-inequalities.
 \end{remark}

\begin{example}\label{ex:1}To begin with, consider an $\mathrm{LE}$-signature such that $\mathcal{F} = \{\wdia,\circ\}$, $\mathcal{G} = \{\wbox\}$, $n_\circ=2$, and $n_\wdia = n_\wbox = \varepsilon_\wdia(1)=\varepsilon_\wbox(1)=\varepsilon_\circ(1)=\varepsilon_\circ(2)=1$. The axiom 
 \begin{equation*} \label{eq:first_equation}
     \Box(\Diamond p \circ p) \circ \Diamond p \fCenter p
 \end{equation*} is definite inductive but not analytic inductive. We run ALBA on it until we reach a polarity-safe form. In the remainder, the very first step of an ALBA run is referred to  as the \emph{first approximation step}.

{{ 
\begin{center}
\begin{tabular}{@{}clr@{}}
& ALBA run computing the inception rule for $\Box(\Diamond p \circ p) \circ \Diamond p \fCenter p$: \\
\hline
    \ &$\Box(\Diamond p \circ p) \circ \Diamond p \leq p$ & \ (1)\\
iff \ & $\forall p \forall \nomi \forall \nomj \forall \cnomm [ \rnomi \leq \Box(\Diamond p \circ p)\ \&\ \rnomj \leq p\ \&\ p\leq \rcnomm \Rightarrow \bnomi \circ \Diamond \bnomj \leq \bcnomm]$ &  (2)\\
iff \ & $\forall \nomi \forall \nomj \forall \cnomm [ \rnomi \leq \Box(\Diamond \rcnomm \circ \rcnomm)\ \&\ \rnomj \leq \rcnomm \Rightarrow \bnomi \circ \Diamond \bnomj \leq \bcnomm]$ &  (3)\\
iff \ & $\forall \nomi \forall \nomj \forall \cnomm [ \forall\cnomn(\Diamond \rcnomm \circ \rcnomm \leq \cnomn \Rightarrow\rnomi\leq\Box\cnomn)\ \&\ \rnomj \leq \rcnomm \Rightarrow \bnomi \circ \Diamond \bnomj \leq \bcnomm]$ &  (4)\\
iff \ & $\forall \nomi \forall \nomj \forall \cnomm [ \forall\cnomn(\forall\nomk \forall\nomh(\nomk\leq\rcnomm\ \&\ \nomh \leq \rcnomm \Rightarrow \Diamond\nomk \circ \nomh \leq \cnomn) \Rightarrow\rnomi\leq\Box\cnomn)\ \&\ \rnomj \leq \rcnomm \Rightarrow \bnomi \circ \Diamond \bnomj \leq \bcnomm]$ &  (5)\\
\end{tabular}
\end{center}
}}

From the last line of the derivation above we can obtain both the first-order correspondent of $\Box(\Diamond p \circ p) \circ \Diamond p \fCenter p$, and the unary inception rule of depth $1$
\begin{center}
\AXC{$Y\fCenter Z$}
\AXC{$\left[X\fCenter \check{\Box}N \right]^\mathcal{R}_{\{N\}}$}
\LL{$R_0$}
\BIC{$X\ \hat{\circ}\ \hat{\Diamond} Y \fCenter Z$}
\DP
\end{center}
where $\mathcal{R}$ is the singleton set containing the rule
\begin{center}
\AXC{$K\fCenter Z$}
\AXC{$H \fCenter Z$}
\LL{$R_1$}
\BIC{$\hat{\Diamond}K\ \hat{\circ}\ H \fCenter N$}
\DP
\end{center}

\noindent Let us  derive  $\Box(\Diamond p \circ p) \circ \Diamond p \fCenter p$ from the rule just obtained (display rules are omitted):

\begin{center}
\begin{tabular}{rc@{}l}
\AXC{$p\fCenter p$}
\AXC{$\left[\pi \right]^\mathcal{R}_{\{N\}}$}
\LL{$R_0$}
\BIC{$\Box(\Diamond p \circ p)\ \hat{\circ}\ \hat{\Diamond} p \fCenter p$}
\UIC{$\Box(\Diamond p \circ p)\ \hat{\circ}\ \Diamond p \fCenter p$}
\UIC{$\Box(\Diamond p \circ p) \circ \Diamond p \fCenter p$}
\DP
&  
\ \ \ \ \ \ \ \ \ \ \ \ 
where $\pi$ is: 
& 
\AXC{$p\fCenter p$}
\AXC{$p \fCenter p$}
\LL{$R_1$}
\BIC{$\hat{\Diamond}p\ \hat{\circ}\ p \fCenter N$}
\UIC{$\Diamond p\ \hat{\circ}\ p \fCenter N$}
\UIC{$\Diamond p \circ p \fCenter N$}
\UIC{$\Box(\Diamond p \circ p) \fCenter \check{\Box}N$}
\DP 
\\
\end{tabular}
\end{center}
\end{example}

\begin{example}\label{ex:2}Consider now the definite inductive but not analytic inductive axiom $p \fCenter \Diamond\Box\Diamond\Box\Diamond p$. We run ALBA on it until we reach a polarity-safe form.

\begin{center}
\begin{tabular}{@{}clr@{}}
& ALBA run computing the inception rule for $p \fCenter \Diamond\Box\Diamond\Box\Diamond p$: \\
\hline
 & $p \fCenter \Diamond\Box\Diamond\Box\Diamond p$ & \ \\
iff & $ \forall p \forall \nomj \forall \cnomm (\rnomj \leq p\ \&\ \Diamond\Box\Diamond\Box\Diamond p \leq \rcnomm \Rightarrow \bnomj \leq \bcnomm)$ &  \\
iff & $ \forall \nomj \forall \cnomm (\Diamond\Box\Diamond\Box\Diamond \rnomj \leq \rcnomm \Rightarrow \bnomj \leq \bcnomm)$ &  \\
iff & $ \forall \nomj \forall \cnomm ( \forall \nomi (\nomi \leq \Box\Diamond\Box\Diamond \rnomj \Rightarrow \Diamond \nomi \leq \rcnomm) \Rightarrow \bnomj \leq \bcnomm)$ &  \\
iff & $ \forall \nomj \forall \cnomm ( \forall \nomi ( \forall \cnomn (\Diamond\Box\Diamond \rnomj \leq \cnomn \Rightarrow \nomi \leq \Box \cnomn) \Rightarrow \Diamond \nomi \leq \rcnomm) \Rightarrow \bnomj \leq \bcnomm)$ &  \\
iff & $ \forall \nomj \forall \cnomm ( \forall \nomi ( \forall \cnomn (\forall \nomk (\nomk \leq \Box\Diamond\rnomj \Rightarrow\Diamond\nomk \leq \cnomn) \Rightarrow \nomi \leq \Box \cnomn) \Rightarrow \Diamond \nomi \leq \rcnomm) \Rightarrow \bnomj \leq \bcnomm)$ &  \\
iff & $ \forall \nomj \forall \cnomm ( \forall \nomi ( \forall \cnomn (\forall \nomk ( \forall \cnomo(\Diamond\rnomj \leq \cnomo \Rightarrow \nomk \leq \Box\cnomo) \Rightarrow\Diamond\nomk \leq \cnomn) \Rightarrow \nomi \leq \Box \cnomn) \Rightarrow \Diamond \nomi \leq \rcnomm) \Rightarrow \bnomj \leq \bcnomm)$ &  \\
\end{tabular}
\end{center}

The last line of the derivation above yields both the first-order correspondent of $p \fCenter \Diamond\Box\Diamond\Box\Diamond p$, and the $0$-ary inception rule $R$ of depth $2$. Notice how the nominal and conominal variables get translated as structure metavariable (when their polarity is positive) or as parameters in a contract (when their polarity is negative).

\begin{center}
\begin{tabular}{rcl}
\AXC{$\left[\hat{\Diamond} I \fCenter M\right]^{\{R'\}}_{\{I\}}$}
\LL{$R$}
\UIC{$J \fCenter M$}
\DP
\  &  \ 
\AXC{$\left[\hat{\Diamond} K \fCenter N\right]^{\{R''\}}_{\{K\}}$}
\LL{$R'$}
\UIC{$I \fCenter \check{\Box}N$}
\DP
\  &  \ 
\AX$\hat{\Diamond} J \fCenter O$
\LL{$R''$}
\UI$K \fCenter \check{\Box}O$
\DP
\\

 & \\
\end{tabular}
\end{center}

Let us show how  $p \fCenter \Diamond\Box\Diamond\Box\Diamond p$ can be derived from the rule just obtained:

\begin{center}
\begin{tabular}{cr@{}cr@{}c}
\AXC{$\left[\pi_1\right]^{\{R'\}}_{\{I\}}$}
\LL{$R$}
\UIC{$\underline{p} \fCenter \Diamond\Box\Diamond\Box\Diamond p$}
\DP
 & $\pi_1:$ & 
\AXC{$\left[ \pi_2\right]^{\{R''\}}_{\{K\}}$}
\LL{$R'$}
\UIC{$I \fCenter \check{\Box}\Diamond\Box\Diamond p$}
\UIC{$I \fCenter \Box\Diamond\Box\Diamond p$}
\UIC{$\hat{\Diamond}I \fCenter \Diamond\Box\Diamond\Box\Diamond p$}
\DP
 & $\ \pi_2:$ & 
\AX$\underline{p} \fCenter p$
\UI$\hat{\Diamond} \underline{p} \fCenter \Diamond p$
\LL{$R''$}
\UI$K \fCenter \check{\Box}\Diamond p$
\UI$K \fCenter \Box\Diamond p$
\UI$\hat{\Diamond}K \fCenter \Diamond\Box\Diamond p$
\DP
\\
\end{tabular}
\end{center}
\end{example}

\begin{example}\label{ex:3}Consider now the definite inductive but not analytic inductive axiom $\Diamond p \star p \fCenter p$, where $\star \in \mathcal{G}$ is such that $\varepsilon_{1,\star} = \varepsilon_{2,\star} = 1$. We run ALBA on it until we reach a polarity-safe form.

\begin{center}
\begin{tabular}{@{}clr@{}}
& ALBA run computing the inception rule for $\Diamond p \star p \fCenter p$: \\
\hline
    \ & $\Diamond p \star p \fCenter p$ & \ \\
iff \ & $ \forall p \forall \nomj \forall \cnomm (\nomj \leq \Diamond p \star p\ \&\  p \leq \cnomm \Rightarrow \nomj \leq \cnomm)$ &  \\
iff \ & $ \forall \nomj \forall \cnomm ( \nomj \leq \Diamond\cnomm\star\cnomm \Rightarrow \nomj \leq \cnomm)$ &  \\
iff \ & $ \forall \nomj \forall \cnomm ( \forall \cnomn_1\forall\cnomn_2(\Diamond \cnomm \leq \cnomn_1\ \&\ \cnomm \leq \cnomn_2 \Rightarrow \nomj \leq \cnomn_1\star\cnomn_2) \Rightarrow \nomj \leq \cnomm)$ &  \\
iff \ & $ \forall \nomj \forall \cnomm ( \forall \cnomn_1\forall\cnomn_2(\forall\nomk(\nomk\leq\cnomm\Rightarrow\Diamond\nomk\leq\cnomn_1))\ \&\ \forall\nomh(\nomh\leq\cnomm\Rightarrow\nomh\leq\cnomn_2) \Rightarrow \nomj \leq \cnomn_1\star\cnomn_2) \Rightarrow \nomj \leq \cnomm)$ & 
\end{tabular}
\end{center}

The last line of the derivation above yields both the first-order correspondent of $\Diamond p \star p \fCenter p$, and the $0$-ary inception rule $R$ of depth $1$.

\begin{center}
\begin{tabular}{rcl}

\AXC{$\left[ J \fCenter N_1 \star N_2\right]^{\{R_1,R_2\}}_{\{N_1,N_2\}}$}
\LL{$R$}
\UIC{$J \fCenter M$}
\DP
\  &  \ 
\AXC{$ K \fCenter M$}
\LL{$R_1$}
\UIC{$\hat{\Diamond} K \fCenter N_1$}
\DP
\  &  \ 
\AXC{$ H \fCenter M$}
\LL{$R_2$}
\UIC{$H \fCenter N_2$}
\DP
\\

 & \\
\end{tabular}
\end{center}

Let us show how $\Diamond p \star p \fCenter p$ can be derived from the rule just obtained.

\begin{center}
\begin{tabular}{cr@{}c}
\AXC{$\left[\pi_1\right]^{\{R_1,R_2\}}_{\{N_1,N_2\}}$}
\LL{$R$}
\UIC{$\Diamond p \star p \fCenter p$}
\DP
 &  $\ \pi_1:$ &  
\AXC{$p \fCenter p$}
\LL{$R_1$}
\UIC{$\hat{\Diamond} p \fCenter N_1$}
\UIC{$\Diamond p \fCenter N_1$}
\AXC{$p \fCenter p$}
\LL{$R_2$}
\UIC{$p \fCenter N_2$}
\BIC{$\Diamond p \star p \fCenter N_1 \ \check{\star}\ N_2$}
\DP
\end{tabular}
\end{center}
\end{example}

\begin{example}\label{ex:4}Consider now the inductive but not analytic inductive axiom 
\[\Diamond(\Box{\lhd}(q\circ r) \wedge \Box(p\star\Box q))\leq{\lhd}\Box(p\wedge r) \vee\Diamond p,\] where $\lhd \in \mathcal{F}$ such that $\varepsilon_\lhd(1) = \partial$ and its Galois-adjoint is denoted as $\blacktriangleleft$. We denote as $/_\star$ and $\setminus_\star$ the residual of $\star$ on the first and second coordinate, respectively. We run ALBA on it until we reach a polarity-safe form.

\begin{center}
\begin{tabular}{@{}clr@{}}
& ALBA run computing the inception rule for $\Diamond(\Box{\lhd}(q\circ r) \wedge \Box(p\star\Box q))\leq\lhd\Box(p\wedge r) \vee\Diamond p$: \\
\hline
    \ & $\Diamond(\Box{\lhd}(q\circ r) \wedge \Box(p\star\Box q))\leq\lhd\Box(p\wedge r) \vee\Diamond p$ & \ \\
iff \ & $ \forall q \forall r\forall p \forall \nomj \forall \cnomm(\nomj \leq \Box{\lhd}(q\circ r)\ \&\ \nomj \leq \Box(p\star\Box q)\ \&\ {\lhd}\Box(p\wedge r)\leq \cnomm\ \&\ \Diamond p \leq \cnomm \Rightarrow \Diamond\nomj \leq \cnomm)$ &  \\
iff \ & $ \forall q \forall r\forall \nomj \forall \cnomm(\nomj \leq \Box{\lhd}(q\circ r)\ \&\ \nomj \leq \Box(\blacksquare\cnomm \star\Box q)\ \&\ {\lhd}\Box(\blacksquare \cnomm\wedge r)\leq \cnomm\Rightarrow \Diamond\nomj \leq \cnomm)$ &  \\
iff \ & $ \forall q\forall \nomj \forall \cnomm(\nomj \leq \Box{\lhd}(q\circ \Diamondblack{\blacktriangleleft} \cnomm)\ \&\ \nomj \leq \Box(\blacksquare\cnomm \star\Box q)\ \&\ \Diamondblack{\blacktriangleleft} \cnomm\leq \blacksquare\cnomm\Rightarrow \Diamond\nomj \leq \cnomm)$ &  \\
iff \ & $ \forall \nomj \forall \cnomm(\nomj \leq \Box{\lhd}(\Diamondblack(\blacksquare\cnomm/_\star\Diamondblack\nomj)\circ \Diamondblack{\blacktriangleleft} \cnomm)\ \&\ \Diamondblack{\blacktriangleleft} \cnomm\leq \blacksquare\cnomm\Rightarrow \Diamond\nomj \leq \cnomm)$ & \\
& & \\
 \ & $ \nomj \leq \Box{\lhd}(\Diamondblack(\blacksquare\cnomm/_\star\Diamondblack\nomj)\circ \Diamondblack{\blacktriangleleft} \cnomm)$ & \\
 iff \ & $ \forall \cnomn(\lhd(\Diamondblack(\blacksquare\cnomm/_\star\Diamondblack\nomj)\circ \Diamondblack{\blacktriangleleft} \cnomm)\leq\cnomn \Rightarrow \nomj \leq \Box\cnomn)$ & \\
  iff \ & $ \forall \cnomn(\forall \cnoml(\Diamondblack(\blacksquare\cnomm/_\star\Diamondblack\nomj)\circ \Diamondblack{\blacktriangleleft} \cnomm \leq \cnoml \Rightarrow \lhd \noml \leq \cnomn) \Rightarrow \nomj \leq \Box\cnomn)$ & \\
\end{tabular}
\end{center}

The derivation above yields both the first-order correspondent of 
\[\Diamond(\Box{\lhd}(q\circ r) \wedge \Box(p\star\Box q))\leq\lhd\Box(p\wedge r) \vee\Diamond p,\] 
\noindent and the following $1$-ary inception rule $R$ of depth $1$.

\begin{center}
\begin{tabular}{rl}
\AXC{$\hat{\Diamondblack}\hat{\blacktriangleleft} M \fCenter \check{\blacksquare}M$}
\AXC{$\left[\nomj \fCenter \check{\Box}N\right]^{\{R'\}}_{\{N\}}$}
\LL{$R$}
\BIC{$\hat{\Diamond}J \fCenter M$}
\DP
\  &  \ 
\AXC{$\hat{\Diamondblack}(\check{\blacksquare}M\hat{/}_{\star}\hat{\Diamondblack}J)\ \hat{\circ}\ \hat{\Diamondblack}\hat{\blacktriangleleft}M \fCenter L$}
\LL{$R'$}
\UIC{$\hat{\lhd} L \fCenter N$}
\DP
\\
\end{tabular}
\end{center}

Let us show how  $\Diamond(\Box{\lhd}(q\circ r) \wedge \Box(p\star\Box q))\leq\lhd\Box(p\wedge r) \vee\Diamond p$ can be derived from the rule just obtained:

\begin{center}
\AXC{$p\fCenter p$}
\UIC{$\hat{\Diamond}p\fCenter \Diamond p$}
\UIC{$\hat{\Diamond}p\fCenter {\lhd}\Box(p\wedge r) \vee \Diamond p$}
\UIC{$p\fCenter \check{\blacksquare}(\lhd\Box(p\wedge r) \vee \Diamond p)$}
\UIC{$p\wedge r\fCenter \check{\blacksquare}(\lhd\Box(p\wedge r) \vee \Diamond p)$}
\UIC{$\Box(p\wedge r)\fCenter \check{\Box}\check{\blacksquare}(\lhd\Box(p\wedge r) \vee \Diamond p)$}
\UIC{$\hat{{\lhd}}\check{\Box}\check{\blacksquare}({\lhd}\Box(p\wedge r) \vee \Diamond p)\fCenter{\lhd}\Box(p\wedge r)$}
\UIC{$\hat{{\lhd}}\check{\Box}\check{\blacksquare}({\lhd}\Box(p\wedge r) \vee \Diamond p)\fCenter{\lhd}\Box(p\wedge r)\vee\Diamond p$}
\doubleLine
\UIC{$\hat{\Diamondblack}\hat{\blacktriangleleft}({\lhd}\Box(p\wedge r) \vee \Diamond p)\fCenter\check{\blacksquare}({\lhd}\Box(p\wedge r)\vee\Diamond p)$}
\AXC{$\left[\pi\right]^{\{R'\}}_{\{N\}}$}
\LL{$R$}
\BIC{$\hat{\Diamond}(\Box{\lhd}(q\circ r) \wedge \Box(p\star\Box q))\fCenter{\lhd}\Box(p\wedge r) \vee\Diamond p$}
\UIC{$\Diamond(\Box{\lhd}(q\circ r) \wedge \Box(p\star\Box q))\fCenter{\lhd}\Box(p\wedge r) \vee\Diamond p$}
\DP
\end{center}
where $\pi$ is:
\begin{center}
{\small
\AXC{$p \fCenter p$}
\UIC{$\hat{\Diamond}p \fCenter \Diamond p$}
\UIC{$\hat{\Diamond}p \fCenter{\lhd}\Box(p\wedge r) \vee\Diamond p$}
\UIC{$p \fCenter\check{\blacksquare}({\lhd}\Box(p\wedge r) \vee\Diamond p)$}

\AXC{$q \fCenter q$}
\UIC{$\Box q \fCenter \check{\Box} q$}
\BIC{$p\star\Box q\fCenter \check{\blacksquare}({\lhd}\Box(p\wedge r) \vee\Diamond p)\ \check{\star}\ \check{\Box}q$}
\UIC{$\Box(p\star\Box q)\fCenter \check{\Box}(\check{\blacksquare}({\lhd}\Box(p\wedge r) \vee\Diamond p)\ \check{\star}\ \check{\Box}q)$}
\UIC{$\Box{\lhd}(q\circ r) \wedge \Box(p\star\Box q)\fCenter \check{\Box}(\check{\blacksquare}({\lhd}\Box(p\wedge r) \vee\Diamond p)\ \check{\star}\ \check{\Box}q)$}
\UIC{$\hat{\Diamondblack}\Diamond(\Box{\lhd}(q\circ r) \wedge \Box(p\star\Box q))\fCenter \check{\blacksquare}({\lhd}\Box(p\wedge r) \vee\Diamond p)\ \check{\star}\ \check{\Box}q$}
\doubleLine
\UIC{$\hat{\Diamondblack}(\check{\blacksquare}({\lhd}\Box(p\wedge r) \vee\Diamond p)\hat{/}_{\star}\hat{\Diamondblack}\Diamond(\Box{\lhd}(q\circ r) \wedge \Box(p\star\Box q)))\fCenter q$}

\AXC{$r\fCenter r$}
\UIC{$p \wedge r\fCenter r$}
\UIC{$\Box(p \wedge r)\fCenter \check{\Box} r$}
\UIC{$\hat{{\lhd}}\check{\Box} r\fCenter {\lhd}\Box(p \wedge r)$}
\UIC{$\hat{{\lhd}}\check{\Box} r\fCenter {\lhd}\Box(p \wedge r) \vee \Diamond p$}
\doubleLine
\UIC{$\hat{\Diamondblack}\hat{\blacktriangleleft}({\lhd}\Box(p \wedge r) \vee \Diamond p)\fCenter r$}
\BIC{$\hat{\Diamondblack}(\check{\blacksquare} ({\lhd}\Box(p\wedge r) \vee\Diamond p) \hat{/}_{\star}\hat{\Diamondblack}(\Box{\lhd}(q\circ r) \wedge \Box(p\star\Box q)))\ \hat{\circ}\ \hat{\Diamondblack}\hat{\blacktriangleleft} ({\lhd}\Box(p\wedge r) \vee\Diamond p) \fCenter q\circ r$}
\LL{$R'$}
\UIC{$\hat{{\lhd}}(q\circ r) \fCenter N$}
\UIC{${\lhd}(q\circ r) \fCenter N$}
\UIC{${\lhd}(q\circ r)\wedge\Box(p\star\Box q) \fCenter N$}
\UIC{$\Box{\lhd}(q\circ r)\wedge\Box(p\star\Box q) \fCenter\check{\Box}N$}
\DP
}
\end{center}

In the derivation above, double lines denote two (or more) consecutive application of a residuation postulate.
\end{example}

\begin{example} \label{ex:5}Finally, consider the definite inductive but not analytic inductive axiom $\Diamond p \fCenter \Diamond\Box\Diamond\Box\Diamond\Box\Diamond p$. We run ALBA on it until we reach a polarity-safe form.
 
\begin{center}
\begin{tabular}{@{}clr@{}}
& ALBA run computing the inception rule for $\Diamond p \fCenter \Diamond\Box\Diamond\Box\Diamond\Box\Diamond p$: \\
\hline
    \ & $\Diamond p \fCenter \Diamond\Box\Diamond\Box\Diamond\Box\Diamond p$ & \ \\
iff \ & $\forall p\forall \nomj \forall \cnomm (\nomj \leq p\ \&\ \Diamond\Box\Diamond\Box\Diamond\Box\Diamond p \leq \cnomm \Rightarrow \Diamond\nomj \leq \cnomm)$ &  \\
iff \ & $\forall \nomj \forall \cnomm (\Diamond\Box\Diamond\Box\Diamond\Box\Diamond \nomj \leq \cnomm \Rightarrow \Diamond\nomj \leq \cnomm)$ &  \\
iff \ & $\forall \nomj \forall \cnomm (\forall \nomi(\nomi \leq \Box\Diamond\Box\Diamond\Box\Diamond \nomj \Rightarrow \Diamond\nomi \leq \cnomm) \Rightarrow \Diamond\nomj \leq \cnomm)$ &  \\
iff \ & $\forall \nomj \forall \cnomm (\forall \nomi( \forall\cnomn(\Diamond\Box\Diamond\Box\Diamond \nomj \leq \cnomn \Rightarrow \nomi \leq \Box\cnomn) \Rightarrow \Diamond\nomi \leq \cnomm) \Rightarrow \Diamond\nomj \leq \cnomm)$ &  \\
iff \ & $\forall \nomj \forall \cnomm (\forall \nomi( \forall\cnomn( \forall \nomk( \nomk \leq \Box\Diamond\Box\Diamond \nomj \Rightarrow \Diamond \nomk \leq \cnomn) \Rightarrow \nomi \leq \Box\cnomn) \Rightarrow \Diamond\nomi \leq \cnomm) \Rightarrow \Diamond\nomj \leq \cnomm)$ &  \\
iff \ & $\forall \nomj \forall \cnomm (\forall \nomi( \forall\cnomn( \forall \nomk( \forall \cnoml(\Diamond\Box\Diamond \nomj \leq \cnoml \Rightarrow \nomk \leq \Box \cnoml)) \Rightarrow \Diamond \nomk \leq \cnomn) \Rightarrow \nomi \leq \Box\cnomn) \Rightarrow \Diamond\nomi \leq \cnomm) \Rightarrow \Diamond\nomj \leq \cnomm)$ &  \\
 &  &  \\
 \ & $ \forall \cnoml(\Diamond\Box\Diamond \nomj \leq \cnoml \Rightarrow \nomk \leq \Box \cnoml)$ &  \\
iff \ & $ \forall \cnoml( \nomh(\nomh \leq \Box\Diamond\nomj \Rightarrow\Diamond\nomh\leq\cnoml) \Rightarrow \nomk \leq \Box \cnoml)$ &  \\
iff \ & $ \forall \cnoml( \nomh( \cnomo(\Diamond\nomj\leq\cnomo\Rightarrow\nomh\leq\Box\cnomo) \Rightarrow\Diamond\nomh\leq\cnoml) \Rightarrow \nomk \leq \Box \cnoml)$ & 
\end{tabular}
\end{center}

The derivation above yields the first-order correspondent of $\Diamond p \fCenter \Diamond\Box\Diamond\Box\Diamond\Box\Diamond p$ together with its equivalent $0$-ary inception rule $R_0$ of depth $3$. 

\begin{center}
\begin{tabular}{rl}

\AXC{$\left[\hat{\Diamond} I \fCenter M\right]^{\{R_1\}}_{\{I\}}$}
\LL{$R_0$}
\UIC{$\hat{\Diamond}J \fCenter M$}
\DP
\  &  \ 
\AXC{$\left[\hat{\Diamond} K \fCenter N\right]^{\{R_2\}}_{\{K\}}$}
\LL{$R_1$}
\UIC{$I \fCenter \check{\Box}N$}
\DP
\\ \\
\AXC{$\left[\hat{\Diamond} H \fCenter L\right]^{\{R_3\}}_{\{H\}}$}
\LL{$R_2$}
\UIC{$K \fCenter \check{\Box}L$}
\DP 
\ & \ 
\AX$\hat{\Diamond} J \fCenter O$
\LL{$R_3$}
\UI$H \fCenter \check{\Box}O$
\DP\\
\end{tabular}
\end{center}

Let us show how  $\Diamond p \fCenter \Diamond\Box\Diamond\Box\Diamond\Box\Diamond p$ can be derived from the rule just obtained:

\begin{center}
\begin{tabular}{r@{}cr@{}c}
 & 
\AXC{$\left[\pi_1\right]^{\{R_1\}}_{\{I\}}$}
\LL{$R_0$}
\UIC{$\hat{\Diamond} p \fCenter \Diamond\Box\Diamond\Box\Diamond\Box\Diamond p$}
\UIC{$\Diamond p \fCenter \Diamond\Box\Diamond\Box\Diamond\Box\Diamond p$}
\DP  
 & $\pi_1:$ & 
\AXC{$\left[ \pi_2\right]^{\{R_3\}}_{\{H\}}$}
\LL{$R_1$}
\UIC{$I \fCenter \check{\Box}\Diamond\Box\Diamond\Box\Diamond p$}
\UIC{$I \fCenter \Box\Diamond\Box\Diamond\Box\Diamond p$}
\UIC{$\hat{\Diamond}I \fCenter \Diamond\Box\Diamond\Box\Diamond\Box\Diamond p$}
\DP \\ 

& & & \\
  
$\ \pi_2:$ & 
\AXC{$\left[ \pi_3\right]^{\{R_2\}}_{\{K\}}$}
\LL{$R_2$}
\UIC{$K \fCenter \check{\Box}\Diamond\Box\Diamond p$}
\UIC{$K \fCenter \Box\Diamond\Box\Diamond p$}
\UIC{$\hat{\Diamond}K \fCenter \Diamond\Box\Diamond\Box\Diamond p$}
\DP
 & $\ \pi_3:$ &  
\AX$p \fCenter p$
\UI$\hat{\Diamond} p \fCenter \Diamond p$
\LL{$R_3$}
\UI$H \fCenter \check{\Box}\Diamond p$
\UI$H \fCenter \Box\Diamond p$
\UI$\hat{\Diamond}H \fCenter \Diamond\Box\Diamond p$
\DP 
 \\
\end{tabular}
\end{center}
\end{example}

\begin{remark}\label{rem:coincidence}
    The reader may have noticed that in all the examples presented in this section, all the metavariables contained in the aim of a contract (cf.~Definition \ref{def:depthrules}) come either from the corresponding endsequent of the rule or from one of the sets $\mathcal{X}$s. This is not a coincidence, but is common to all the inception rules coming from a polarity-safe ALBA output. This property could be exploited in future works to give a syntactic characterization of the inception rules obtained via ALBA from inductive inequalities.
\end{remark}

\section{Cut elimination}\label{sec:cut_elimination}
In the present section, via a Belnap-style metatheorem which builds on \cite{frittella2014multi}, we prove that if all the (additional) structural  rules of an inception calculus are ALBA-generated (cf.~Section \ref{ssec:examples}), then the calculus is cut-eliminable. Throughout this section, $X$ and $Y$ denote structure metavariables.
\begin{definition} \label{def:cut_depth}
    The \emph{inception depth} of a derivation $\pi$, denoted as $d(\pi)$, is the supremum of the depths of the rules instantiated in its nodes (cf.~Definition \ref{def:depthrules}). 

    Consider a cut between the sequents $X\fCenter A$ and $A\fCenter Y$, and let $\pi_1$ (resp.~$\pi_2$) be the derivation with endsequent $X\fCenter A$ (resp.~$A\fCenter Y$). We define the \emph{inception depth} of the cut to be $d(\pi_1) + d(\pi_2)$. If both $\pi_1$ and $\pi_2$ do not contain instances of the cut rule (cf.~discussion below Notation \ref{notation_dream}), we say that the cut between $X \fCenter A$ and $A\fCenter Y$ is an \emph{uppermost cut}.
\end{definition}

\begin{definition} \label{def:congruent_structures}
Let $\pi$ be a derivation. Two structures in $\pi$ are \emph{locally congruent} if they instantiate the same structure metavariable in an application of a rule, including structure metavariables occurring in contracts at any depth. The \emph{congruence} relation is defined as the reflexive and transitive closure of the aforementioned relation.
\end{definition}

\begin{remark}
    In an inception display calculus, locally congruent parameters are not necessarily `close' to each other: in Example \ref{ex:2}, the antecedent of $p \fCenter \Diamond\Box\Diamond\Box\Diamond p$ is locally congruent to two other structures in $\pi_2$ (underlined in the example), but to no substructure in $\pi_1$.
\end{remark}

\begin{definition}\label{def:contract_sub}
    Let us specify how substitution for succedent parameters works in the presence of higher-order assumptions. The substitution for precedent parameters is defined analogously. If the application $R$ of a rule is of the form

\begin{center}
    \AXC{$S_1[A]^{suc}_1\ \ \cdots\ \ S_n[A]^{suc}_n\ \ \ \ \left[ T_{n+1}\right]^{\mathcal{R}_1}_{\mathcal{X}_1}[A]^{suc}_{n+1}\ \ \cdots\ \ \left[ T_{n+m}\right]^{\mathcal{R}_m}_{\mathcal{X}_m}[A]^{suc}_{n+m}$}
    \RL{,}
    \UIC{$S_0[A]^{suc}$}
    \DP
\end{center}

\noindent where $S_i$ and $T_i$ are metavariables for arbitrary sequents, and $[A]^{suc}_i$ represents all and only the occurrences of $A$ in $S_i$ (when $1\leq i\leq n$) or in $T_i$ (when $n<i\leq n+m$) which are congruent to an occurrence of $A$ in the conclusion, which may or may not be present in $S_0$ (cf.~Example \ref{ex:sub}). Then $R[Y/A]^{suc}$ is recursively defined as

\begin{center}
    \AXC{$S_1[Y/A]^{suc}_1\ \ \cdots\ \ S_n[Y/A]^{suc}_n\ \ \ \ \left[ T_{n+1}\right]^{\mathcal{R}_1}_{\mathcal{X}_1}[Y/A]^{suc}_{n+1}\ \ \cdots\ \ \left[ T_{n+m}\right]^{\mathcal{R}_m}_{\mathcal{X}_m}[Y/A]^{suc}_{n+m}$}
    \RL{,}
    \UIC{$S_0[Y/A]^{suc}$}
    \DP
\end{center}

\noindent where we let $\left[ T_{n+j} \right]^{\mathcal{R}_j}_{\mathcal{X}_j}[Y/A]^{suc}_{n+j}$ be $\left[ T_{n+j} [Y/A]^{suc}_{n+j}\ \right]^{\mathcal{\overline{R}}_j}_{\mathcal{X}_j}$, with $\overline{\mathcal{R}}_j = \{R'[Y/A]^{suc}\ |\ R'\in\mathcal{R}_j\}$.
\end{definition}

\begin{example} \label{ex:sub}
    Consider the application of the rule $R$ in the derivation of $p \fCenter \Diamond\Box\Diamond\Box\Diamond p$ in Example \ref{ex:2}. Following Definition \ref{def:contract_sub}, the substitution $R[\hat{\Diamondblack}p/p]^{pre}$ is:

\begin{center}
\begin{tabular}{rcl}

\!\!\!\!\!\!
\AXC{$\left[I \fCenter \check{\Box}\Diamond\Box\Diamond p\right]^{\overline{\{R'\}}}_{\{I\}}$}
\LL{$R[\hat{\Diamondblack} p/p]$}
\UIC{$\hat{\Diamondblack}p \fCenter \Diamond\Box\Diamond\Box\Diamond p$}
\DP
\  &  \!\!\!\!\!\!\!\!\!
\AXC{$\left[K \fCenter \check{\Box}\Diamond p\right]^{\overline{\{R''\}}}_{\{K\}}$}
\LL{$R'[\hat{\Diamondblack} p/p]$}
\UIC{$I \fCenter \check{\Box}\Diamond\Box\Diamond p$}
\DP
\  &  \!\!\!\!\!\!\!\!\!
\AX$\hat{\Diamond} \hat{\Diamondblack} p \fCenter \Diamond p$
\LL{$R''[\hat{\Diamondblack} p/p]$}
\UI$K \fCenter \check{\Box}\Diamond p$
\DP
\\

 & \\
\end{tabular}
\end{center}
    
\end{example}

\begin{remark} \label{rem:old_C6_C7}
Augmenting a proper display calculus with an inception rule could result in a calculus which is no longer properly displayable, because inception rules are not closed under simultaneous substitution of structural variables. Hence, conditions C$_6$ and C$_7$ defined in Appendix \ref{ap:display} fail in general. However, it will be shown that analytic inception rules satisfy weaker versions of C$_6$ and C$_7$, which make it still possible to prove a Belnap-style cut elimination. Indeed, consider the rule instance
\begin{center}
\begin{tabular}{r l}
\AXC{$p\fCenter p$}
\AXC{$\left[\Box(\Diamond p \circ p) \fCenter \check{\Box}N \right]^{\{R_1\}}_{\{N\}}$}
\LL{$R_0$}
\BIC{$\Box(\Diamond p \circ p)\ \hat{\circ}\ \hat{\Diamond} p \fCenter p$}
\DP \ \ \ 
&
\ \ \ \AXC{$p\fCenter p$}
\AXC{$p \fCenter p$}
\RL{$R_1$}
\BIC{$\hat{\Diamond}p\ \hat{\circ}\ p \fCenter N$}
\DP
\end{tabular}
\end{center}
from Example \ref{ex:1}, where the contract of $R_0$ is fulfilled by the following derivation $\pi$. 
\begin{center}
    \AXC{$p\fCenter p$}
\AXC{$p \fCenter p$}
\LL{$R_1$}
\BIC{$\hat{\Diamond}p\ \hat{\circ}\ p \fCenter N$}
\UIC{$\Diamond p\ \hat{\circ}\ p \fCenter N$}
\UIC{$\Diamond p \circ p \fCenter N$}
\UIC{$\Box(\Diamond p \circ p) \fCenter \check{\Box}N$}
\DP 
\end{center}
If we try to perform the substitution $[\hat{\Diamond}p/\Box(\Diamond p\circ p)]$ we get
\begin{center}
\AXC{$p\fCenter p$}
\AXC{$\left[\hat{\Diamond} p \fCenter \check{\Box}N \right]^\mathcal{R}_{\{N\}}$}
\LL{$R_0$}
\BIC{$\hat{\Diamond}p\ \hat{\circ}\ \hat{\Diamond} p \fCenter p$}
\DP
\end{center}
but now $\pi[\hat{\Diamond}p/\Box(\Diamond p\circ p)]$ is not even a legitimate derivation\footnote{More precisely, $\pi[\hat{\Diamond}p/\Box(\Diamond p\circ p)]$ is the derivation $\pi$ where all the instances of $\Box(\Diamond p\circ p)$ that are congruent (cf.~Definition \ref{def:congruent_structures}) with the instance of $\Box(\Diamond p\circ p)$ in the endsequent are substituted with $\hat\wdia p$.} anymore, which means that the result of the substitution is not an instance of the rule, and we need to look for a different witness (cf.~Notation \ref{notation_dream}).

\end{remark}

\begin{definition}
\label{def:ProperLabelledCalculi}
An inception display calculus is  \emph{proper} if it satisfies conditions C$_1$-C$_8$ (cf.~Appendix \ref{ap:display}), where C$_6$ and C$_7$ are modified as follows:

\noindent \textbf{C$_6$ (resp.~C$_7$): Quasi-closure under substitution for succedent (resp.~precedent) parts.\;} Each rule without higher-order assumptions is closed under simultaneous substitution of arbitrary structures for congruent formulas occurring in succedent (resp.~precedent) position.
Furthermore, if $A \fCenter Y$ (resp.~$Y \fCenter A$) is derivable and $\left[ \Pi\fCenter\Sigma\right]^\mathcal{R}_\mathcal{X}$ is a contract that holds by virtue of a cut-free derivation $\pi$, then the contract $(\left[ \Pi\fCenter\Sigma\right]^\mathcal{R}_\mathcal{X})[Y/A]^{suc}$ (resp.~$(\left[ \Pi\fCenter\Sigma\right]^\mathcal{R}_\mathcal{X})[Y/A]^{pre}$) can be fulfilled by a derivation $\pi'$ such that every instance of a cut rule in $\pi'$ has premises $A \fCenter Y$ (resp.~$Y \fCenter A$) and a sequent which is display-equivalent to some sequent from $\pi[Y/A]^{suc}$ (resp.~$\pi[Y/A]^{pre}$).

\end{definition}
In what follows we will omit the superscripts indicating the (preceding or succedent) position of parameters when the argument does not depend on it.
\begin{lemma} \label{lem:C6}
Let $R$ be an  inception rule generated by the translation of a polarity-safe ALBA output of a (definite) inductive axiom. For any application  $R[A]$ of $R$ in which $A$ occurs parametrically (also) in the endsequent of $R[A]$, any contract-rule in $R[Y/A]$ (cf.~Definition \ref{def:contract_sub})   
is of the form

\begin{center}
    \AXC{$S_1[Y/A]\ \ \cdots\ \ S_n[Y/A]\ \ \ \ \left[ T_{n+1}\right]^{\mathcal{R}_1}_{\mathcal{X}_1}[Y/A]\ \ \cdots\ \ \left[ T_{n+m}\right]^{\mathcal{R}_m}_{\mathcal{X}_m}[Y/A]$}
    \RL{,}
    \UIC{$S_0$}
    \DP
\end{center}

\noindent that is, no substitution takes place in the endsequent of any contract-rule of $R[Y/A]$ at any depth.

\end{lemma}

\begin{proof}
Recall (cf.~Remark \ref{footn:polarity-safe}) that we refer to the nominal and conominal variables introduced with first approximation in ALBA as the extractors. As observed in Remark \ref{footn:polarity-safe}, after reaching the polarity-safe shape (cf.~Definition \ref{def:polarity_safe}), the extractors occurring on the left-hand side of the main quasi-inequality  always occur within inequalities in negative position.
This implies that, after obtaining $R$ by translating the polarity-safe ALBA output of the given inductive inequality, the metavariables of $R$ corresponding to the extractors (which are exacty the metavariables occurring in the conclusion of  $R$) cannot occur in the conclusions of any contract-rules  at any depth in $R$, since the inequalities in negative position can only be translated either as premises of rules or as aims of contracts. Hence, the substitution $[Y/A]$ cannot involve the endsequent of any  contract rule at any depth.  
\end{proof}

\begin{thm} \label{thm:inception_preserves_C6_C7}
Inception display calculi with ALBA-generated inception rules satisfy conditions C$_1$-C$_8$.
\end{thm}

\begin{proof}
C$_1$ holds because structure variables which occur fresh within dreams are never instantiated.
 Our convention throughout the paper is that congruent parameters are denoted by the same letter. Hence,  C$_2$ holds. C$_3$ holds because, by first approximation, each extractor only occurs once in the conclusion of the main quasi-inequality generating each structural rule.
C$_4$ immediately follows from the definition of polarity-safe shape (cf.~Definition \ref{def:polarity_safe}).
 Conditions C$_5$ and C$_8$ hold trivially for the same reason they hold in proper display calculi. 
 
 Let us show  C$_6$ (the proof of C$_7$ being similar). It is enough to show that,  if $A \fCenter Y$ (resp.~$Y \fCenter A$) is derivable and
 $\left[ \Pi\fCenter\Sigma\right]^\mathcal{R}_\mathcal{X}$ is a contract which is witnessed by a cut-free derivation $\pi$, then $(\left[ \Pi\fCenter\Sigma\right]^\mathcal{R}_\mathcal{X})[Y/A]^{suc}$ can be fulfilled by a derivation $\pi'$ such that every instance of a cut rule in $\pi'$ has premises $A \fCenter Y$ and a sequent which is display-equivalent to some sequent from $\pi[Y/A]^{suc}$. Recall that $(\left[ \Pi\vdash \Sigma \right]^{\mathcal{R}}_{\mathcal{X}})[Y/A]^{suc}$ is defined as $\left[ (\Pi\vdash \Sigma) [Y/A]^{suc}\ \right]^{\mathcal{\overline{R}}}_{\mathcal{X}}$, with $\overline{\mathcal{R}} = \{R'[Y/A]^{suc}\ |\ R'\in\mathcal{R}\}$.

\begin{center}
\AXC{\ \ \ \ \ $\vdots$ \raisebox{1mm}{$\pi''$}}
\noLine
\UIC{$(S_i)[A]^{suc}$}
\noLine
\UIC{\ \ \ \ \ \ \ \ \ \ $\vdots$ \raisebox{1mm}{\fns display}}
\noLine
\UI$\Pi_i\fCenter A$
\AX$A\fCenter Y$
\RL{Cut}
\BI$\Pi_i\fCenter Y$
\noLine
\UIC{\ \ \ \ \ $\vdots$ \raisebox{1mm}{\fns display}}
\noLine
\UIC{$S_i[Y/A]^{suc}$}
\RL{$\overline{R} \in \overline{\mathcal{R}}$}
\UIC{$S$}
\noLine
\UIC{\ \ \ $\ddots\vdots\iddots$ \raisebox{1mm}{$\pi'$}}
\noLine
\UI$(\Pi \fCenter \Sigma)[A]^{suc}$
\noLine
\UIC{\ \ \ \ \ \ \ \ \ \ $\vdots$ \raisebox{1mm}{\fns display}}
\noLine
\UI$\Pi'\fCenter A$
\AX$A\fCenter Y$
\RL{Cut}
\BI$\Pi'\fCenter Y$
\noLine
\UIC{\ \ \ \ \ \ \ \ \ \ $\vdots$ \raisebox{1mm}{\fns display}}
\noLine
\UI$(\Pi \fCenter \Sigma)[Y/A]^{suc}$
\DP
\end{center}
Reasoning upwards from the conclusion, let us show that we can always construct the required witness of $(\left[ \Pi\fCenter\Sigma\right]^\mathcal{R}_\mathcal{X})[Y/A]^{suc}$ from a cut-free witness $\pi$ (see the sketch above) of $\left[ \Pi\fCenter\Sigma\right]^\mathcal{R}_\mathcal{X}$. In order to derive $(\Pi \fCenter \Sigma)[Y/A]^{suc}$, the first step is to reach $\Pi \fCenter \Sigma$ bottom-up by  displaying $Y$ in succedent position, applying cut with $A\vdash Y$  and applying the display rules in reverse order. This is always possible because we assumed $A \fCenter Y$ is derivable.  From $\Pi \fCenter \Sigma$ we can proceed upward following $\pi$, until we reach the endsequent $S$ of a rule $R \in \mathcal{R}$ (step $\pi'$). By Lemma \ref{lem:C6}, we can apply the corresponding rule $\overline{R} \in \overline{\mathcal{R}}$, the premises of which are of the form $S_i[Y/A]^{suc}$, where $S_i$ are the premises of $R$. From each $S_i[Y/A]^{suc}$ which is a premise, we can reach the corresponding $S_i$ by proceeding analogously to what we did to reach $\Pi\vdash \Sigma$; namely, by displaying $Y$, then cutting with $A$ and hence applying display rules in the reverse order. Each such $S_i$ is derivable using the rules in $\overline{\mathcal{R}}$, by recursively applying the procedure just discussed, and relying on $\pi$.  
The case in which $S_i[Y/A]^{suc}$ is a higher-order assumption is treated recursively. 
The process just described terminates since, by assumption, every derivation has finite dreams.
Since Cut was applied only between $A \fCenter Y$ and sequents (which are display-equivalent to some sequents) from $\pi[Y/A]^{suc}$,  C$_6$ is verified.
\end{proof}

\begin{thm}
Inception display calculi with ALBA-generated inception rules are cut-eliminable.
\end{thm}

\begin{proof}
Let $d$ be the inception depth of an uppermost cut (cf.~Definition \ref{def:cut_depth}). Let us show that the aforementioned uppermost cut can be eliminated. The proof proceeds by induction on $d$.
If $d = 0$ then the cut can be eliminated as in proper display calculi.

As to inductive case, if $d = k+1$, conditions C$_2$-C$_8$ allow us to follow every move (both principal and parametric) of the cut-elimination metatheorem (cf.~\cite{wansing1998displaying,LEframes,ucaaptt}), introducing new cuts inside the dream of some of the contracts while performing a parametric move. Let the cut sequents be $X \fCenter A$ and $A \fCenter Y$. Since every new cut introduced in a dream $\left[ \pi\right]$ by virtue of conditions C$_6$ and C$_7$ has either $X \fCenter A$ and sequents  from $\pi[X/A]^{suc}$ (modulo display equivalence) or $A \fCenter Y$ and sequents from $\pi[Y/A]^{pre}$ (modulo display equivalence) as cut sequents, each of these cuts has a cut inception depth less than or equal to $k$. Hence,  we can apply the induction hypothesis and recursively eliminate the newly introduced cuts.
\end{proof}

\begin{example} \label{exe:inception_cut_elimination}
Consider the $\mathrm{LE}$-logic in the language $\mathcal{F} = \{\wdia,\circ\}$, $\mathcal{G} = \{\wbox\}$ given by the basic $\mathrm{LE}$-logic extended with the axioms 

\[\wdia p \vdash \wbox(\wdia p \circ p) \quad\quad \Box(\Diamond p \circ p) \circ \Diamond p \fCenter p.\] 
The residuals of the connective $\circ$ are denoted as $/$ and $\backslash$ for the first and second coordinate, respectively; the residual of the connective $\wdia$ (resp.~$\wbox$) is denoted as $\blacksquare$ (resp.~$\Diamondblack$). As we have seen in Example \ref{ex:1}, the second axiom is equivalent to the analytic inception rule
\begin{center}
\begin{tabular}{cc}
\AXC{$Y\fCenter Z$}
\AXC{$\left[X\fCenter \WBOX N \right]^{\{R_1\}}_{\{N\}}$}
\LL{$R_0$}
\BIC{$X\ \hat{\circ}\ \hat{\Diamond} Y \fCenter Z$}
\DP
&
\AXC{$K\fCenter Z$}
\AXC{$H \fCenter Z$}
\LL{$R_1$}
\BIC{$\hat{\Diamond}K\ \hat{\circ}\ H \fCenter N$}
\DP
\end{tabular}
\end{center}
    \begin{center}
    \end{center}
while the first axiom is equivalent\footnote{We used the method described in \cite{ucaaptt} to obtain the rule. We omit its computation because it is not relevant for the current example.} to the analytic structural rule
\begin{center}
    \AXC{$\hat{\wdia} X\ \hat{\circ}\ X \vdash Y$}
    \LL{$R$}
    \UIC{$\hat{\wdia} X \vdash \check{\wbox} Y$}
    \DP
\end{center}
Consider the following derivation, where double lines denote repeated applications of display postulates and invertible logical introduction rules:
\begin{center}
    \AXC{$p\vdash p$}
    \UIC{$\hat{\wdia}p\vdash\wdia p$}
    \AXC{$p\vdash p$}
    \BIC{$\hat{\wdia} p\ \hat{\circ}\ p \vdash \wdia p \circ p$}
    \LL{$R$}
    \UIC{$\hat{\wdia} p \vdash \check{\wbox} (\wdia p \circ p)$}
    \UIC{$\hat{\wdia} p \vdash \wbox (\wdia p \circ p)$}
    
    \AXC{$p \vdash p$}
    \AXC{$\left[\pi\right]^{\{R_1\}}_{\{N\}}$}
    \RL{$R_0$}
    \BIC{$\wbox (\wdia p \circ p)\ \hat{\circ}\ \hat{\wdia} p \vdash p $}
    \UIC{$\wbox (\wdia p \circ p) \vdash \hat{\wdia} p\ \check{\backslash}\ p $}
    \LL{Cut}
    \BIC{$\hat{\wdia} p \vdash  \hat{\wdia} p\ \check{\backslash}\ p$}
    \UIC{$\hat{\wdia} p\ \hat{\circ}\ \hat{\wdia} p \vdash p$}
    \DP
\end{center}
where $\pi$ is the derivation
\begin{center}
    \AXC{$p \vdash p$}
    \AXC{$p \vdash p$}
    \LL{$R_1$}
    \BIC{$\hat{\wdia} p\ \hat{\circ}\  p\vdash N$}
    \doubleLine
    \UIC{$\wdia p \circ p\vdash N$}
    \UIC{$\wbox (\wdia p \circ p)\vdash \check{\wbox}N$}
    \DP
\end{center}
The cut formula $\wbox (\wdia p \circ p)$ is not principal on the right premise and it is not introduced by a logical rule, thus, we can eliminate the cut with a parametric step, substituting $\hat{\wdia} p$ for the cut formula in the derivation of the left cut premise. We obtain the derivation
\begin{center}
    \AXC{$p\vdash p$}
    \AXC{$\left[\pi'\right]^{\{R_1\}}_{\{N\}}$}
    \LL{$R_0$}
    \BIC{$\hat{\wdia} p\ \hat{\circ}\ \hat{\wdia}p \vdash p$}
    \UIC{$\hat{\wdia} p \vdash  \hat{\wdia} p\ \check{\backslash}\ p$}
    \UIC{$\hat{\wdia} p\ \hat{\circ}\ \hat{\wdia} p \vdash p$}
    \DP
\end{center}
where, following the proof in Theorem \ref{thm:inception_preserves_C6_C7}, $\pi'$ is
\begin{center}
    \AXC{$p\vdash p$}
    \UIC{$\hat{\wdia}p\vdash\wdia p$}
    \AXC{$p\vdash p$}
    \BIC{$\hat{\wdia} p\ \hat{\circ}\ p \vdash \wdia p \circ p$}
    \LL{$R$}
    \UIC{$\hat{\wdia} p \vdash \check{\wbox} (\wdia p \circ p)$}
    \UIC{$\hat{\wdia} p \vdash \wbox (\wdia p \circ p)$}
    \AXC{$p \vdash p$}
    \AXC{$p \vdash p$}
    \LL{$R_1$}
    \BIC{$\hat{\wdia} p\ \hat{\circ}\  p\vdash N$}
    \doubleLine
    \UIC{$\wdia p \circ p\vdash N$}
    \UIC{$\wbox (\wdia p \circ p)\vdash \check{\wbox}N$}
    \LL{Cut}
    \BIC{$\hat{\wdia} p \vdash \check{\wbox} N$}
    \DP
\end{center}
Notice how we eliminated the original cut at the price of creating another one inside the dream $\pi'$. But $\pi'$ has depth $0$, therefore we can safely eliminate this new cut following the standard non-inception metatheorem of cut elimination. 
\end{example}

\section{Conclusion} \label{sec:concl}

In this paper, the proof-theoretic framework for LE-logics developed in \cite{ucaaptt,ChnGrePalTzi21,LEframes}, aimed at capturing classes of axiomatic extensions of normal lattice-based logics by means of analytic calculi,  is extended from the class of analytic inductive inequalities \cite[Definition 2.14]{ChnGrePalTzi21} to the  class of inductive inequalities \cite[Definition 3.4]{conradie2019algorithmic} in arbitrary  LE-signatures. The approach adopted in the present paper builds on  \cite{ucaaptt,ChnGrePalTzi21,LEframes,DDG2022}, and hinges on the algorithmic generation of analytic structural rules equivalently capturing inductive LE-axioms using the algorithm ALBA, originally introduced for computing the first-order correspondents of inductive LE-axioms in arbitrary signatures, within a purely semantically motivated research line \cite{CoGhPa14}. Applying this approach to arbitrary inductive axioms results in the algorithmic generation of structural rules of display calculi which feature higher-order assumptions of a particular form discussed in Definition  \ref{def:depthrules}, referred to as {\em contracts} of positive depth. Since fulfilling this type of higher-order assumptions involves the existence of derivations of certain sequents, which might themselves involve rules featuring contracts as assumptions, structural rules with higher-order assumptions are referred to as {\em inception rules}, and {\em inception display calculi} have been introduced as proper display calculi augmented with inception rules. In this paper, we mainly focus on inception calculi augmented with ALBA-generated rules, for which we prove a Belnap-style cut elimination metatheorem, thus taking a first step towards extending the fully fledged theory developed in \cite{ucaaptt,ChnGrePalTzi21,LEframes} to inductive LE-axioms.

The inception framework opens up the possibility of adopting proof analysis methods for the study of certain inductive axiomatic extensions of $\mathrm{LE}$-logics. 
In \cite[Definition 7.2.5]{DeDomenico}, a subclass of inception rules (the {\em analytic inception rules}) is introduced, and proper inception calculi are defined as proper display calculi augmented with analytic inception rules. We conjecture that proper inception calculi of finite depth are precisely those that capture the LE-logics axiomatized by inductive axioms, and that these calculi are syntactically complete \cite{ChnGrePalTzi21}. Furthermore, we conjecture that the notion of proper inception calculus introduced in \cite{DeDomenico} coincides with the one proposed in the present paper.
Other interesting directions include comparing inception display calculi with labelled calculi expanded with systems of rules \cite{Negri14}, and considering inception calculi admitting non-well-founded nesting of dreams to capture fixed point logics, possibly extending the reach of standard cyclic sequent calculi \cite{BahWeh22}.

\bibliography{ref}
\bibliographystyle{plain}

\begin{appendix}
    
\section{Proper display calculi}\label{ap:display}

\begin{definition}
A proof system enjoys the {\em display property} if, for every sequent $\Pi \fCenter \Sigma$ and every substructure $\Upsilon$ of either $\Pi$ or $\Sigma$, the sequent $\Pi \fCenter \Sigma$ can be equivalently transformed, using the rules of the system, into a sequent which is either of the form $\Upsilon \fCenter \Xi$ or of the form $\Xi \fCenter \Upsilon$, for some structure $\Xi$. In the first case, $\Upsilon$ is displayed in {\em precedent position}, and in the second case, $\Upsilon$ is displayed in {\em succedent position}.  

If a structure $\Upsilon$ occurs in the sequent $\Pi \fCenter \Sigma$ in precedent (resp.~succedent) position we write $(\Pi \fCenter \Sigma)[\Upsilon]^{pre}$ (resp.~$(\Pi \fCenter \Sigma)[\Upsilon]^{suc}$).
\end{definition}
	
\begin{definition}
\label{def:ParametersCongruenceHistoryTree}
{\em Specifications} are instantiations of structure metavariables in the statement of a rule $R$. The {\em parameters} of $r \in R$ are {\em substructures} of instantiations of structure metavariables in the statement of $R$. A formula instance is {\em principal} in an inference $r \in R$ if it is not a parameter in the conclusion of $r$. Structure occurrences in an inference $r \in R$ are in the (symmetric) relation of {\em local congruence} if they instantiate the same metavariable in the statement of $R$. Therefore, the local congruence is a relation between specifications.
\end{definition}

We recall the conditions C$_1$-C$_8$ defining a {\em proper display calculus}. Our presentation closely follows \cite[Section 2.2]{ucaaptt}.

	\paragraph*{ C$_1$: Preservation of formulas.} This condition requires each formula occurring in a premise of a given inference to be a subformula of some formula in the conclusion of that inference. 
This condition is not included in the list of sufficient conditions of the cut elimination metatheorem (cf.~Theorem \ref{thm:meta}), but, in the presence of cut elimination, it guarantees the \emph{subformula property} of a proof system.
	\paragraph*{ C$_2$: Shape-alikeness.}
	This condition requires that locally congruent specifications are occurrences of the same structure. 

	\paragraph*{C$_3$: Non-proliferation.} 
 Condition C$_3$ requires that, for every inference, each of its specifications is locally congruent to at most one specification in the conclusion of that inference. 
	\paragraph*{C$_4$: Position-alikeness.} This condition bans any rule in which a (sub)structure in precedent (resp.~succedent) position in a premise is locally congruent to a (sub)structure in succedent (resp.~precedent) position in the conclusion.
	
	\paragraph*{C$_5$: Display of principal constituents.} This condition requires that any principal occurrence 
 be always either the entire antecedent or the entire consequent part of the sequent in which it occurs.
	
	%
	
	\paragraph*{C$_6$: Closure under substitution for succedent parameters.} This condition requires each rule to be closed under simultaneous substitution of arbitrary structures for formulas which are congruent parameters occurring in succedent position.
	Condition C$_6$ ensures, for instance, that if the following inference is an application of the rule $R$:
	
	\begin{center}
		\AX$(\Pi \fCenter \Sigma) \big([A]^{suc}_{i} \,|\, i \in I\big)$
		\RightLabel{$R$}
		\UI$(\Pi' \fCenter \Sigma') [A]^{suc}$
		\DisplayProof
	\end{center}
	
	\noindent and $\big([A]^{suc}_{i} \,|\, i \in I\big)$ represents all and only the occurrences of $A$ in the premiss which are congruent to the occurrence of $A$ in the conclusion, 
	then also the following inference is an application of the same rule $R$:
	
	\begin{center}
		\AX$(\Pi \fCenter \Sigma) \big([\Upsilon/A]^{suc}_{i} \,|\, i \in I\big)$
		\RightLabel{$R$}
		\UI$(\Pi' \fCenter \Sigma') [\Upsilon/A]^{suc}$
		\DisplayProof
	\end{center}
	
	\noindent where the structure $\Upsilon$ is substituted for $A$.

	\paragraph*{C$_7$: Closure under substitution for precedent parameters.} This condition requires each rule to be closed under simultaneous substitution of arbitrary structures for formulas which are congruent parameter occurring in precedent position.
	Condition C$_7$ can be understood analogously to C$_6$, relative to formulas in precedent position. 

	\paragraph*{C$_8$: Eliminability of matching principal constituents.}
	
	This condition  requests a standard Gentzen-style checking, which is limited to the case in which both cut formulas are {\em principal}, i.e.\ each of them has been introduced with the last rule application of each corresponding subdeduction. 

	\begin{thm} (cf.\ \cite[Section 3.3, Appendix A]{Wan02})
		\label{thm:meta}
		Any calculus satisfying conditions C$_2$, C$_3$, C$_4$, C$_5$, C$_6$, C$_7$, C$_8$ enjoys cut elimination. If C$_1$ is also satisfied, then the calculus enjoys the subformula property.
	\end{thm}
	
\begin{prop}\label{def:analytic}
		(cf.\ \cite[Definition 3.13]{CiRa14}) A structural rule which satisfies conditions C$_1$-C$_7$ is an \emph{analytic structural rule}.
	\end{prop}

\end{appendix}
\end{document}